\documentclass[a4paper]{amsart}
\usepackage{graphicx,amsthm, enumerate}  
\usepackage{mathrsfs}
\usepackage[all]{xy}

\usepackage[hypertexnames=false]{hyperref} 

\newtheorem{thm}{Theorem}[section]
\newtheorem{prop}[thm]{Proposition}
\newtheorem{lem}[thm]{Lemma}

\newtheorem{defn}[thm]{Definition}

\theoremstyle{definition} 
\newtheorem{disc}[thm]{Discussion}

\newtheorem*{ass}{Assertion} 
\newtheorem*{assum}{Assumption}

\numberwithin{equation}{section}
\theoremstyle{definition}
\newtheorem{rem}[thm]{Remark}
\newtheorem{ex}[thm]{Example}

\newcommand{\G}{\Gamma}
\newcommand{\id}{\rm id}

\newcommand{\N}{\mathbb{N}}
\newcommand{\C}{\mathbb{C}}

\newcommand{\Z}{\mathbb{ Z}}

\newcommand{\Hom}{\operatorname{Hom}}
\newcommand{\hra}{\hookrightarrow}

\newcommand{\KK}{\rm K}

\newcommand{\R}{\mathbb{ R}}

\newcommand{\rk}{\operatorname{rk}}

\newcommand{\x}{\times}

\newcommand{\CP}{\mathbb{ C}{\rm P}}

\def\C{\mathord{\mathbb C}}
\def\D{\mathord{\mathbb D}}
\def\G{\mathord{\mathbb G}}
\def\H{\mathord{\mathbb H}}
\def\K{\mathord{\mathbb K}}

\def\O{\mathord{\mathbb O}}
\def\R{\mathord{\mathbb R}}
\def\P{\mathord{\mathbb P}}
\def\Z{\mathord{\mathbb Z}}
\def\N{\mathord{\mathbb N}}
\def\SS{\mathord{\mathbb S}}

\def\3{{\ss}}
\def\2{\frac{1}{2}}
\def\x{\times}
\def\.{\cdot}
\def\d{\partial}
\def\o{\circ}

\def\<{\langle}
\def\>{\rangle}

\def\id{{\rm id}}
\def\diag{\mathop{\rm diag}\nolimits}
\def\Hom{\mathop{\rm Hom}\nolimits}
\def\End{\mathop{\rm End}\nolimits}

\def\Fix{\mathop{\rm Fix}\nolimits}
\def\Map{\mathop{\rm Map}\nolimits}

\def\Hopf{\mathop{\rm Hopf}\nolimits}

\def\Iso{\mathop{\rm Iso}\nolimits}

\def\Im{\mathop{\rm Im\,}\nolimits}
\def\Re{\mathop{\rm Re\,}\nolimits}

\def\Cl{\mathop{\rm Cl}\nolimits}
\def\Spin{\mathop{\rm Spin}\nolimits} 
\def\SO{\mathop{\rm SO}\nolimits} 

\def\tr{\mathop{\rm tr}\nolimits} 

\def\U{{\rm U}} 
\def\Or{{\rm O}}
\def\Sp{{\rm Sp}} 
\def\Gl{{\rm Gl}} 
\def\SU{{\rm SU}} 

\def\SSpan{\mathop{\rm Span\,}\nolimits}

\def\ind{\mathop{\rm ind}\nolimits}
\def\trace{\mathop{\rm trace\,}\nolimits}
\def\nhalf{\begin{smallmatrix} \underline{n} \cr 2 \end{smallmatrix}}

\def\g{\mathfrak{g}}
\def\t{\mathfrak{t}}

\def\so{\mathfrak{so}}
\def\i{{\sf i}}

\def\'{\textrm{''}}

\def\sn{\smallskip\noindent}           
\def\ms{\medskip\noindent}
\def\beq{\begin{equation}}
\def\eeq{\end{equation}}
\def\bea{\begin{eqnarray}}
\def\eea{\end{eqnarray}}
\def\bsm{\left(\begin{smallmatrix}}
\def\esm{\end{smallmatrix}\right)}
\def\bpm{\begin{pmatrix}}
\def\epm{\end{pmatrix}}

\begin{document}

\title[Bott-Thom isomorphism, Hopf bundles and Morse theory]{Bott-Thom isomorphism, Hopf bundles \\ and Morse theory}
\dedicatory{Dedicado à memória de Manfredo do Carmo}
\author{Jost-Hinrich Eschenburg}
\author{Bernhard Hanke}
\address{Institut f\"ur Mathematik, Universit\"at Augsburg, D-86135 Augsburg, Germany}
\email{eschenburg@math.uni-augsburg.de}
\email{hanke@math.uni-augsburg.de}
\thanks{The second named author was supported by the SPP 2026 {\em Geometry at Infinity} funded by the DFG}
 
\begin{abstract} Based on  Morse theory for the energy functional on path spaces   we develop a deformation theory for mapping spaces of  spheres into orthogonal groups. 
This is used to show that these mapping spaces are weakly homotopy equivalent, in a stable range, to mapping spaces associated to orthogonal Clifford representations. 

Given an oriented Euclidean bundle $V \to X$ of rank divisible by four over a finite complex $X$ we derive a stable decomposition result for vector bundles over  the sphere bundle $\SS( \R \oplus V)$  in terms of vector bundles and Clifford module bundles over $X$.

After passing to topological K-theory these results imply classical Bott-Thom isomorphism theorems.
\end{abstract}


\date{\today}
\maketitle

\tableofcontents

\section{Introduction}
In their seminal paper on Clifford modules \cite{ABS} Atiyah, Bott and Shapiro describe a far-reaching interrelation between the representation theory of Clifford algebras  and topological K-theory.
This point of  view inspired Milnor's exposition \cite{M} of Bott's proof of the periodicity theorem for the homotopy groups of the orthogonal group. 
The unique flavor of Milnor's approach is that a very peculiar geometric structure (centrioles in symmetric spaces) which is related to algebra (Clifford representations) leads to basic results in topology, via Morse theory on path spaces.

In the paper at hand we rethink Milnor's approach and investigate how far his methods can be extended. 
In fact, they allow dependence on arbitrary many extrinsic local parameters. 
Thus we may replace the spheres in Milnor's computation of homotopy groups by sphere bundles over any finite CW-complex.
Among others this leads to a geometric perspective of Thom isomorphism theorems in topological K-theory. 

This  interplay of algebra, geometry and topology is characteristic for the mathematical thinking of Manfredo do Carmo.
We therefore believe that our work may be a worthwhile contribution to his memory. 

Recall that a Euclidean vector bundle $E$ of rank $p$ over a sphere $\SS^n$ can be described by its {\it clutching map} $\phi : \SS^{n-1}\to \SO_p$. 
In fact, over the upper and lower hemisphere  $E$ is the trivial bundle  $\underline{\R}^p$, and $\phi$ identifies the two fibers $\R^p$ along the common boundary $\SS^{n-1}$ as in the following picture. 

\ms\hskip 3.5cm
\includegraphics{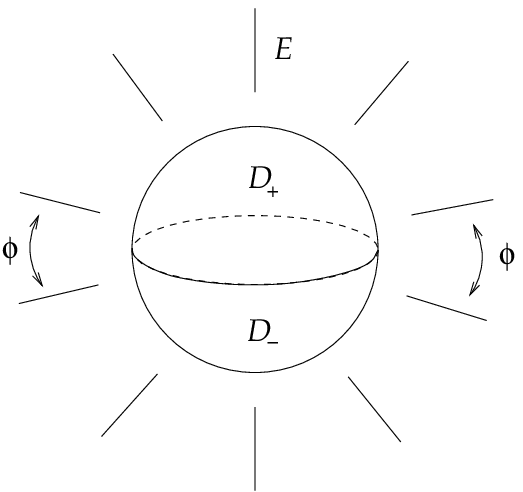}

\medskip
Milnor in his book on Morse theory \cite{M} describes a deformation procedure that can be used to simplify these clutching maps $\phi$.
The main idea in \cite{M} is viewing the sphere as an iterated suspension and the map $\phi$ as an iterated path family in $\SO_p$ with prescribed end points, and then Morse theory for the energy functional on each path space is applied.
However in a strict sense, Morse theory is not applied but {\em avoided}: it is shown that the non-minimal critical points (geodesics) have high index, so they do not obstruct the deformation of the path space onto the set of minima (shortest geodesics) via the negative gradient flow of the energy. 
Thus the full path space is deformed onto the set of shortest geodesics whose midpoint set can be nicely described in terms of certain totally geodesic submanifolds $P_j$ (``centrioles''). 
In fact there is a chain of iterated centrioles $\SO_p\supset P_1\supset P_2\supset\cdots$ such that the natural inclusion of $P_{j}$ into the path space of $P_{j-1}$ is $d$-connected for some large $d$ and for all $j$ (that is, it induces an isomorphism in homotopy groups $\pi_k$ for $k<d$ and a surjection on $\pi_d$). 
This is sufficient for Milnor's purpose to understand the topology of the path spaces in order to compute the stable homotopy groups of $\SO_p$ (Bott periodicity). 

In \cite{EHa} we went one step further and deformed the whole map $\phi$ into a special form: the restriction of a certain linear map $\phi_o : \R^n \to \R^{p\x p}$.
The latter defines a module structure on $\R^p$ for the Clifford algebra $\Cl_{n-1}$, which turns the given bundle $E$ into the Hopf bundle for this Clifford module. 
As shown in \cite{EHa} this leads to  a conceptual proof of  \cite[Theorem (11.5)]{ABS}, expressing the coefficients of  topological K-theory in terms of Clifford representations, and thus  gives a positive response to the remark in \cite[page 4]{ABS}: ``It is to be hoped that Theorem (11.5) can be given a more natural and less computational proof''. 

In the present paper we will put this deformation process into a family context, aiming at a description  of vector bundles over sphere bundles in terms of Clifford representations. 
More specifically, let $V \to X$ be a Euclidean vector bundle over a finite CW-complex $X$ and let $\hat V = \SS(\R \oplus V) \to X$ be the sphere bundle of the direct sum bundle $\R \oplus V$. 
It is a sphere bundle with two distinguished antipodal sections $(\pm 1, 0)$. 
Similar as before a vector bundle $\mathscr{E} \to \hat V$ can be constructed by a fiberwise clutching function along the ``equator spheres'' $\SS(V)$ in each fibre, and one may try to bring this clutching function into a favorable shape by a fiberwise deformation process similar as the one employed in \cite{EHa}.

We will realize this program if $V$ is oriented and of rank divisible by four in order to derive bundle theoretic versions of classical Bott-Thom isomorphism theorems in topological K-theory.
For example let $\rk V = 8m$ and assume that $V\to X$ is equipped with a spin structure.
Let $\mathscr{S} \to \hat V$ be the spinor Hopf bundle associated to the chosen spin structure on $V$ and the unique (ungraded)  irreducible $\Cl_{8m}$-representation, compare Definition \ref{spinorhopf}.
Then each vector bundle $\mathscr{E} \to\hat V$  is  -- after addition of trivial line bundles and copies of $\mathscr{S}$ --  isomorphic to a bundle of the form $E_0\oplus ( E_1\otimes \mathscr{S})$,  where $E_0,E_1$ are vector bundles over $X$. 
Moreover the stable isomorphism types of $E_0$ and $E_1$ are determined by the stable isomorphism type of $\mathscr{E}$, see Remark \ref{83}.
In K-theoretic language this amounts to the classical Thom isomorphism theorem in orthogonal K-theory, compare part (a) of Theorem \ref{ThomClassic}.

Atiyah in his book on K-theory \cite[p.\,64]{A} proved an analogous statement for complex vector bundles $\mathscr{E} \to \hat L$ where $L\to X$ is a Hermitian line bundle and $\hat L =\P(\C\oplus L) = \SS(\R \oplus L) \to X$ is the complex projective bundle with fibre  $\CP^1 =  \SS^2$. 
Now the clutching map of $\mathscr{E}$ is defined on the circle bundle $\SS(L)$ and can be described fiberwise by Fourier polynomials with values in $\Gl_p(\C)$.
Then the higher Fourier modes are removed by some deformation on $\Gl_p(\C)$ after enlarging $p = {\rm rk}(\mathscr{E})$;  only the first (linear) Fourier mode remains.
In our case of higher dimensional $\SS^n$-bundles, Fourier analysis is no longer available.
However, using Milnor's ideas, we still can linearize the clutching map along every fibre. 
But linear maps from the sphere to $\SO_p$ are nothing else than Clifford representations (cf.\ Proposition \ref{clifrep}).

Our paper is organized as follows. 
In Section \ref{cliffordr} we recall some notions from the theory of Clifford modules. 
Section \ref{centrioles} relates the theory of Clifford modules to iterated centrioles in symmetric spaces. 
This setup, which implicitly underlies the argument  in \cite{M}, provides a convenient and conceptual frame for our later arguments. 

A reminder of the Morse theory of the energy functional on path spaces in symmetric spaces  is provided in Section \ref{paths} along the lines in \cite{M}. 
This is accompanied by some explicit index estimates for non-minimal geodesics in Section \ref{index}. 
Different from  \cite{M} we avoid curvature computations using totally geodesic spheres instead.

After these preparations  Section \ref{defmapspace} develops a  deformation theory for pointed mapping spaces $\Map_*(\SS^k , \SO_p)$, based on an iterative use of Morse theory on path spaces in symmetric spaces. 
When $\R^p$ is equipped with a $\Cl_k$-representation, $\Map_*(\SS^k , \SO_p)$ contains the subspace of affine Hopf maps associated to Clifford sub-representations on $\R^p$ (compare Definitions \ref{HopfCliff}, \ref{hopfrho}). 
Our Theorem \ref{mindef} gives conditions under which this inclusion is highly connected. 

Section \ref{spinor} recalls the construction of vector bundles over sphere bundles by clutching data and provides some examples.
The central part of our work is Section \ref{difficult}, 
where we show that if $V \to X$ is an oriented vector bundle of rank divisible by four,
then vector bundles $\mathscr{E} \to  \hat V$ are, after stabilization, sums of bundles which arise from $\Cl(V)$-module bundles over $X$ by the clutching construction and  bundles pulled back from $X$. 
We remark that up to this point our argument is not using topological K-theory. 

The final Section \ref{CliffordThom} translates the results of Section \ref{difficult} into a K-theoretic setting and derives the $\Cl(V)$-linear Thom isomorphism theorem  \ref{clthom} in this language. 
This recovers  Karoubi's Clifford-Thom isomorphism theorem \cite[Theorem IV.5.11]{KaroubiBook} in the special case of oriented vector bundles  $V \to X$ of rank divisible by four.
In this respect we provide a geometric approach to this important result, which is proven in \cite{KaroubiBook} within the theory of Banach categories; see  Discussion \ref{Discussion} at the end of our paper for more details. 
Together with the representation theory of Clifford algebras it also implies the classical Thom isomorphism theorem for orthogonal K-theory.
Finally, for completeness of the exposition we mention the analogous periodicity theorems for unitary and symplectic K-theory, which are in part difficult to find in the literature. 

\section{Recollections on Clifford modules} \label{cliffordr}

Let $(V, \<\ ,\ \>)$ be a Euclidean vector space. 
Recall that the Clifford algebra $\Cl(V)$ is the $\R$-algebra generated by all elements of $V$ with the relations $vw + wv = -2\<v,w\>\, 1$ for all $v,w\in V$, or equivalently, for any orthonormal basis $(e_1,\dots,e_n)$ of $V$,
\beq \label{clifford}
	e_ie_j + e_je_i = -2\delta_{ij} \, .
\eeq
For $V = \R^n$ with the standard Euclidean structure we write $\Cl_n := \Cl(\R^n)$.

Let $\K \in \{ \R , \C, \H\}$. 
An (ungraded)  {\em $\Cl(V)$-representation} is a  
$\K$-module $L$ together with a homomorphism of $\R$-algebras  
\[
       \rho : \Cl(V) \to \End_{\K}(L) \, . 
\]
In other words, $L$ is a $\Cl(V) \otimes \K$-module. 
We also speak of {\em real, complex}, respectively {\em quaternionic} $\Cl(V)$-representations. 

Let  $(e_1, \ldots, e_n)$ be an orthonormal basis of $V$ and put $J_i = \rho(e_i)$. 
Due to \eqref{clifford}  these are anticommuting $\K$-linear complex structures on $L$, that is $J_i^2 = -I$ and $J_iJ_k = -J_kJ_i$ for $i\neq k$. 
We also speak of  a {\em Clifford family} $(J_1, \ldots, J_n)$. 
This implies that all $J_i$ are orientation preserving (for $\K = \C , \H$ this already follows from $\K$-linearity). 

In the following we restrict to real $\Cl(V)$-modules; for complex or quaternionic $\Cl(V)$-modules similar remarks apply. 
We may choose an inner product on $L$ such that $J_i \in \SO(L)$; equivalently all $J_i$ are skew adjoint. 
In this case we also speak of an {\em orthogonal} $\Cl(V)$-representation. 
With the inner product 
\[
    \langle A, B \rangle := \frac{1}{\dim L} \tr(A^T \circ B) 
\]
on $\End(L)$ this implies $J_i \perp J_k $ and $J_i \perp \id_L$ for $ 1 \leq i \neq k \leq n$.
In particular we obtain an isometric linear map  $ \R \oplus V  \to \End(L) $,
\[
     (t, v) \mapsto  t \cdot \id_{L} + \rho(v) \, . 
\]
By the previous remarks it sends the unit sphere $\SS(\R \oplus V) \subset \R \oplus V$ into the special orthogonal group $\SO(L) \subset \End(L)$.

\begin{defn} \label{HopfCliff} 
We call the restriction 
\[
    \mu: \SS(\R \oplus V) \to \SO(L)
\]
the {\em Hopf map} associated to the orthogonal Clifford representation $\rho$. 
\end{defn} 

Isometric linear maps $\R \oplus V \to \End(L)$ are in one-to-one correspondence with isometric embeddings $\SS(\R \oplus V) \to \SS( \End(L))$ onto great spheres.
This leads to the following geometric characterization of Clifford representations.

\begin{prop} \label{clifrep}
Let $L$ be a Euclidean vector space and let 
\[
     \mu : \SS(\R \oplus V)  \to  \SS(\End(L))
 \]
 be an isometric embedding as a great sphere, which satisfies 
 \[
      \mu  ( \SS(\R \oplus V))  \subset \SO(L) \, , \quad  \mu(1,0) = \id \, . 
 \]
 Then $\mu$ is the Hopf map of an orthogonal Clifford representation $\Cl(V) \to \End(L)$. 
\end{prop}

\begin{proof} 
By assumption $\mu$ is the restriction of a linear map $\R \oplus V \to \End(L)$. 
If $A,B \in {\rm image}(\mu) \subset \SO(L)$, then $A+B \in \R \cdot \SO(L)$ by assumption, hence 
\[
   (A+B)^T \cdot (A+B) = t \cdot \id 
\]
for some  $t \in \R$. 
Furthermore 
\[
    (A+B)^T \cdot (A+B) = A^T A + B^T B +  A^T B + B^T A = 2 \cdot I  + A^T B + B^T A \, , 
\]
and hence 
\[
   A^T B + B^T A = s \cdot \id 
\]
with $s = t-2$. 

Taking the trace on both sides we have $s = 2\langle A,B \rangle$. 
If $A = I$ and $B \perp I$, this implies $B + B^T = 0$ . 
For $A , B \perp \id$ we hence get the Clifford relation $AB + BA = -2\langle A,B \rangle \cdot \id$.
\end{proof} 
 
\medskip The structure of the real representations of $\Cl_n$ is well known (cf.\ \cite[p.\,28]{LM}). 
They are direct sums of irreducible representations $\rho_n$.
These are unique and faithful when $n \not\equiv 3\mod 4$.
Otherwise there are two such $\rho_n$, which are both not faithful and differ by an automorphism of $\Cl_n$.
The corresponding modules $S_n$ and algebras $C_n := \rho_n(\Cl_n)$ are as follows. 

\begin{thm} {\rm (Periodicity theorem for Clifford modules)} \label{perclif}
\beq \label{Mk}
\begin{matrix}
	n & 0 & 1 & 2 & 3 & 4 & 5 & 6 & 7 & 8 	\cr
	S_n &\R&\C&\H&\H&\H^2&\C^4&\O&\O&\O^2   \cr
	s_n & 1& 2& 4& 4& 8  & 8  & 8& 8& 16	\cr
	C_n &\R&\C&\H&\H&\H(2) &\C(4)&\R(8)&\R(8)&\R(16)
\end{matrix}\
\begin{matrix} 
	8+k \cr \ \O^2\otimes S_k \cr 16s_k\cr\ \R(16) \!\otimes\!C_k
\end{matrix}
\eeq 
where $s_n = \dim S_n$.
Here $\K(p)$ denotes the algebra of $(p\x p)$-matrices over $\K$.
For $n=3$ and $n=7$, the two different module structures on $S_n =\K$ for $\K = \H,\O$  are generated by the left and the right multiplications, respectively, with elements of  the ``imaginary'' subspace  $\R^n = \Im\K =\R^\perp\subset\K$. 

The action of $(x, \xi)\in\R^{k+8} = \R^k \oplus \O$ on  $S_{k+8} = \O^2\otimes S_k = (\O\otimes S_k)^2$  is given by 
	$$
	\rho_{k+8}(x,\xi) = \bpm x & -L(\bar\xi)^T \cr L(\xi) & -x \epm 
	$$
where $L(\xi) = L(\xi)\otimes \id_{S_k}$ denotes the left translation on $\O$,
and where $x \in \R^k \subset \Cl_k$ acts on $S_k = 1\otimes S_k$ by $\rho_k$. 
\end{thm}

\section{Poles and centrioles} \label{centrioles}

Clifford modules bear a close relation to the geometry of symmetric spaces. 
Let $P$ be a Riemannian symmetric space: for any $p\in P$ there is an isometry $s_p$ of $P$ which is an involution having $p$ as an isolated fixed point.
Two points $o,p\in P$ will be called {\it poles}  if $s_p = s_o$. 
The notion was coined for the north and south pole of a round sphere, but there are many other spaces with poles; e.g.\ $P = \SO_{2n}$ with $o = I$ and  $p = -I$, or the Grassmannian $P=\G_n(\R^{2n})$ with $o = \R^n$ and $p = (\R^n)^\perp$. 
Of course, pairs of poles are mapped onto pairs of poles by isometries of $P$.

A geodesic $\gamma$ connecting poles $o = \gamma(0)$ and $p = \gamma(1)$  is reflected into itself at $o$ and $p$ and hence it is closed with period $2$.

\ms
\hskip 4.5cm
\includegraphics[scale=.7]{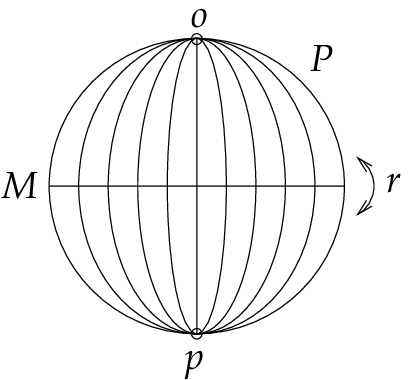}

\ms Now we consider the {\it midpoint set} $M$ between poles $o$ and $p$,
$$
  M = \{m=\gamma\bsm \underline 1 \cr 2 \esm:  
  \gamma\textrm{ shortest geodesic in $P$ with }\gamma(0) = o,\ \gamma(1) = p\}.
$$
For the sphere $P=\SS^n$ with north pole $o$, this set would be the equator. 
In general, $M$ need not be connected, but it is still the fixed point set of a reflection (order-two isometry) $r$ on $P$.\footnote{There is a covering $\pi : P \to \check P$ with $\pi(p) = \pi(o)$ where $\check P = \{s_p:p\in P\} \subset Iso(P)$, and $r = \gamma \o s_o$ where $\gamma$ is the deck transformation of $\pi$ with $\gamma(o) = p$.}
Hence the connected components of $M$, called {\it (minimal) centrioles} \cite{CN}, are totally geodesic subspaces of $P$ -- otherwise short geodesic segments $\sigma$ in the ambient space $P$  with end points in a component of $\Fix(r)$ would not be unique:

\ms
\hskip 4cm
\includegraphics{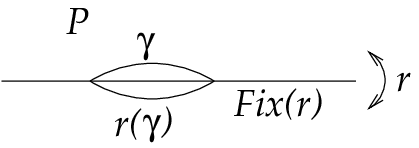}

\ms
Each such midpoint $m = \gamma(\2)$ determines its geodesic $\gamma$ uniquely, and thus the set of minimal geodesics can be replaced with $M$:  if there is another geodesic $\tilde\gamma$ from $o$ to $p$ through $m$, it can be made shorter by cutting the corner at $m$, thus it is not minimal: \label{corner}

\ms\hskip 4cm
\includegraphics{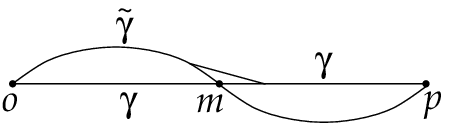}

\medskip There are chains of minimal centrioles (centrioles in centrioles):
\beq \label{P}
	P \supset P_1 \supset P_2 \supset \cdots 
\eeq 
Peter Quast \cite{Q1,Q2} classified all such chains starting from a compact simple Lie group  $P = G$ with at least 3 steps. 
The result is \eqref{table1} below.
The chains 1,2,3 are introduced in Milnor's book \cite{M}.

\beq \label{table1}
\begin{matrix}
\textrm{No.} & G & P_1 & P_2 & P_3 & P_4 & \textrm{restr.}	\cr
\hline
1 & ({\rm S})\Or_{4n} & \SO_{4n}/\U_{2n} & \U_{2n}/\Sp_n & \G_m(\H^n) & \Sp_m & m=\nhalf \cr
2 & ({\rm S})\U_{2n} & \G_n(\C^{2n}) & \U_n & \G_m(\C^n) & \U_m & m=\nhalf \cr
3 & \Sp_n & \Sp_n/\U_n & \U_n/\O_n & \G_m(\R^n) & \SO_m & m=\nhalf \cr
4 & \Spin_{n+2} & {\rm Q}_n & (\SS^1\!\x\!\SS^{n-1})/\!\pm &\SS^{n-2} & \SS^{n-3} & n\geq3 \cr
5 & {\rm E}_7 & {\rm E}_7/(\SS^1{\rm E}_6) & \SS^1 {\rm E}_6/ {\rm F}_4 & \O\P^2 & - \cr
&&
\end{matrix}
\eeq

\noindent
By $\G_m(\K^n)$ we denote the Grassmannian of $m$-dimensional subspaces in $\K^n$ for $\K \in \{\R,\C,\H\}$. 
Further, ${\rm Q}_n$ denotes the complex quadric in $\C\P^{n+1}$,  which is isomorphic to the real Grassmannian $\G^+_2(\R^{n+2})$ of oriented 2-planes in $\R^{n+2}$,  and $\O\P^2$ is the octonionic projective plane ${\rm F}_4/\Spin_9$.

A chain is extendible beyond $P_k$ if and only if $P_k$ contains poles again. 
E.g.\ among the Grassmannians $P_3=\G_m(\K^n)$  only those of half dimensional subspaces ($m = \frac{n}2$) enjoy this property: Then $(E,E^\perp)$ is a pair of poles for any $E\in\G_{n/2}(\K^n)$,  and the corresponding midpoint set is the group $\Or_{n/2}, \U_{n/2}, \Sp_{n/2}$ since its elements are the graphs of orthogonal  $\K$-linear maps $E\to E^\perp$, see figure below.

\hskip 4cm
\includegraphics{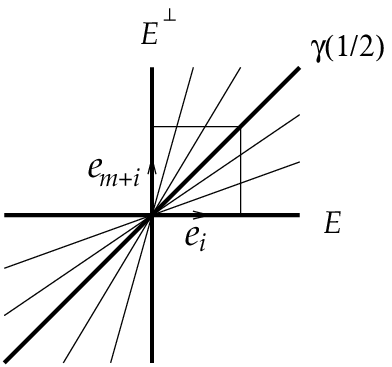}

For compact connected matrix groups $P=G$ containing $-I$, there is a linear algebra interpretation for the iterated minimal centrioles $P_j$. We only consider classical groups.

\begin{thm} \label{thmcentrioles} 
Let $L=\K^p$ with $\K\in\{\R,\C,\H\}$ and $G\in\{\SO_p,\U_p,\Sp_p\}$ with $p$ even in the real case.
Then a chain of minimal centrioles 
\[
  G \supset P_1\supset\dots\supset P_k
\]
 corresponds to a $\Cl_k$-representation $J_1,\dots,J_k$ on $L$ with $J_j\in G$, and each $P_j$ is the connected component through $J_j$ of the set
\beq \label{Pj}
	\hat P_j = \{J\in G: J^2=-I,\, JJ_i = -J_iJ  \textrm{ for } i<j\}.
\eeq
\end{thm}

\proof 
A geodesic $\gamma$ in $G$ with $\gamma(0) = I$ is a {\it one-parameter subgroup}, a Lie group homomorphism\footnote
	{Let $\gamma :\R\to G$ be a smooth group homomorphism. 
	Its curvature vector field $\eta=\nabla_{\gamma'}\gamma'$ is $\gamma$-invariant,
	$\eta(t) = \gamma(t)_*\eta(0)$. 
	When $\eta\neq 0$, the neighbor curve $\gamma_s(t) = \exp_{\gamma(t)}(s\eta(t))$ 		is another $\gamma$-orbit, $\gamma_s(t) = \gamma(t)g_s$. 
	Deforming in the curvature vector direction shortens a curve, thus $\gamma_s$ is 		shorter than $\gamma$ on any finite interval. 

	\sn
	\hskip 5cm \includegraphics{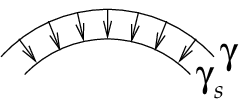}

	\noindent But on the other hand $\gamma_s = R(g_s)\gamma$ 
	has the same length as $\gamma$ since $R(g_s)$ is an isometry of $G$. 
	Thus $\eta = 0$.} 
$\gamma : \R \to G$. When $\gamma(1) = -I$, then $\gamma(\2) = J$ is a {\it complex structure}, $J^2 = -I$. 
Thus the midpoint set $\hat P_1$ is the set of complex structures in $G$. 

By induction hypothesis, we have anticommuting complex structures
$J_i\in P_i$ for $i\leq j$, and $P_j$ is the connected component through $J_j$
of the set $\hat P_j$ as in \eqref{Pj}. Suppose that also 
\beq \label{-Jj}
	-J_j\in P_j.
\eeq 
Consider a shortest geodesic
$\gamma$ from $J_j$ to $-J_j$ in $P_j$. Put $J = \gamma(\2)\in P_j$. Thus $J$ anticommutes
with $J_i$ for all $i<j$. It remains to show that $J$ anticommutes with $J_j$, too.
Since $P_j$ is totally geodesic, $\gamma$ is a geodesic in $G$, hence $\gamma = \gamma_oJ_j$ where $\gamma_o$ is a one-parameter group with $\gamma_o(1) = -I$ which again implies that $J_o := \gamma_o(\2)$ is a complex structure. But also $J\in P_j$ is a complex structure, $J^{-1} = -J$, and since $J = \gamma(\2) = J_oJ_j$, this means $J_jJ_o = -J_oJ_j$. Thus both $J_o$ and $J = J_oJ_j$  anticommute with $J_j$,  hence $J\in \hat P_{j+1}$ as defined in \eqref{Pj}. 

Vice versa, let $J\in \hat P_{j+1}$, that is $J\in G$ is a complex structure anticommuting with $J_1,\dots,J_j$. Then $J_o := JJ_j^{-1}\in G$ is a complex structure which anticommutes with $J_j$ and commutes with $J_i$, $i<j$. Further, from $J_o^{-1} = -J_o$ and $J_o^T = J_o^{-1}$
we obtain $J_o^T = -J_o$, 
thus $J_o \in G\cap\g\subset\End_{\K}(L)$ where $\g$ denotes the Lie algebra of $G$ (here we use 
$G = SO(L)\cap\End_{\K}(L)$).
Putting $\gamma_o(t) = \exp(t\pi J_o)$ we define a geodesic $\gamma_o$ in $G$ from $\gamma_o(0) = I$ to $\gamma_o(1) = -I$ via $\gamma(\2) = J_o$. 
In fact this is shortest in $G$, being a great circle in the plane spanned by $I$ and $J_o$.
Further, the geodesic $\gamma = \gamma_oJ_j$ from $J_j$ to $-J_j$  is contained in $P_j$, due to the subsequent Lemma \ref{JjeA} (applied to $A = \pi J_o$), and it is shortest in $P_j$ (even in the ambient space $G$). 
Thus $J$ is contained in the midpoint set of $(P_j,J_j)$. 
\endproof

\begin{lem} \label{JjeA}
Let $A\in\g$. Then $\exp(tA) J_j \in P_j$ for all $t\in\R$ if and only if
\beq \label{anticom}
	\textrm{$A$ anticommutes with $J_j$ and commutes with $J_i$ for $i<j$.}
\eeq
\end{lem}

\proof \cite[p.\,137]{M} For generic $t\in\R$ we have: 
$A$ anticommutes with $J_j$ $\iff$ $J_j^{-1}\exp(tA)J_j = \exp(-tA)$ $\iff$ 
$\exp(tA)J_j$ $=  J_j\exp(-tA) = -(\exp(tA)J_j)^{-1}$\\ $\iff$ (A): $\exp(tA)J_j$\, is a complex structure.\\ 
Further, $A$ commutes with $J_i$ $\iff$ $J_i$ commutes with $\exp(tA)$\\ 
$\iff$  (B): $\exp(tA)J_j$ anticommutes with $J_i$ (with $i<j$).\\
(A) and (B) $\iff$ $\exp(tA)J_j\in P_{j}$.
\endproof

\begin{rem}
The proof of Theorem \ref{thmcentrioles} shows that the induction step can be carried through as long \eqref{-Jj} holds.
This condition limits the length $k$ of the chain of minimal centrioles.
\end{rem}

\section{Deformations of path spaces} \label{paths}

Minimal centrioles in $P$ are also important from a topological point of view: Since they represent the set of shortest geodesics from $o$ to $p$, they form tiny  models of the path space of $P$. 
This is shown using Morse theory of the energy function on the path space \cite{M}.

\begin{defn} \label{dconnected}
An inclusion $A\subset B$ of topological spaces is called {\em $d$-connected} for some $d\in\N$ if for any $k\leq d$ and any continuous map $\phi : (\D^k,\d \D^k) \to (B,A)$ there is a continuous deformation $\phi_t : \D^k \to B$, $0 \leq t \leq 1$, with 
\[
     \phi_0 = \phi \, , \quad  \phi_1(\D^k) \subset A \, , \quad \phi_t |_{ \partial \D^k} \text{ constant in } t. 
\]
This implies the same property for $(\D^k,\d\D^k)$ replaced by a finite CW-pair $(X,Y)$ with $\dim Y < \dim X \leq d$ by induction over the dimension of the cells in $X$ not contained in $Y$ and homotopy extension (see Corollary 1.4 in [\cite{Br}, Ch. VII]).
\end{defn}

Let $(P,o)$ be a pointed symmetric space and $p$ a pole of $o$. 
Let $\Omega := \Omega(P; o,p)$  be the space of all continuous paths $\omega : [0,1] \to P$ with $\omega(0) = o$ and $\omega(1) = p$ (for short we say $\omega : o\leadsto p$ in $P$), equipped with the compact-open topology (uniform convergence). 
Furthermore let $\Omega^0(P;o,p) \subset \Omega(P; o,p)$ be the subspace of shortest geodesics $\gamma: o\leadsto p$.
The following theorem is essentially due to Milnor \cite{M}.

\begin{thm} \label{short}
Let 
\[
    \Omega_* \subset \Omega(P; o, p)
\]
be a connected component of $\Omega(P; o, p)$  and 
\[
   \Omega^0_* \subset \Omega_* 
\]
be the subspace of minimal geodesics $\gamma : o \leadsto p$ in $\Omega_*$.
Let $d+1$ be the smallest index of all non-minimal geodesics $o\leadsto p$ in $\Omega_*$.

Then the inclusion $\Omega_*^0 \subset \Omega_*$ is $d$-connected. 
\end{thm}

We recall the main steps of the proof. 
The basic idea is using the {\em energy function}
\beq \label{energy}
	E(\omega) = \int_0^1 |\omega'(t)|^2 dt.
\eeq 
for each path $\omega\in\Omega_*$ which is $H^1$ (almost everywhere differentiable with square-integable derivative). 
Applying the gradient flow of $-E$  we may shorten all $H^1$-paths simultaneously to minimal geodesics.\footnote
	{When $\omega$ is reparametrized proportional to arc length 
	(which does not increase energy),
	we have $E(\omega) = |\omega'|^2 = L(\omega)^2$. 
	Thus $L$ is minimized when $E$ is minimized.}

Since the energy is not defined on all of $\Omega_*$,  we will apply this flow only on the subspace of {\em geodesic polygons} $\Omega_n\subset\Omega_*$ for large $n\in\N$ (to be chosen later),  where each such polygon $\omega \in \Omega_n$ has its vertices at $\omega(k/n)$, $k= 0,\dots,n$, and the connecting curves are shortest geodesics. 
For any $r\in\N$ with $n|r$ we have {$\Omega_n\subset\Omega_r$}. 
Furthermore $\Omega^0_* \subset \Omega_n$ for all $n$.

\begin{lem} \label{72}
For all $k \geq 0$ there exists an $n$ such that the inclusion $\Omega_n \subset  \Omega_*$ is $k$-connected.
\end{lem}

\begin{proof} 
Let $\phi : \D^k \to \Omega^*$ with $\phi(\d \D^k) \subset \Omega_n$.
We consider $\phi$ as a continuous map $\D^k \x[0,1]\to P$. 
This is  equicontinuous by compactness of $\D^k \x[0,1]$. 

Let $R$ be the convexity radius on $P$, which means that for any $q\in P$ and any $q',q''\in B_R(q)$ the shortest geodesic between $q'$ and $q''$ is unique and contained in $B_R(q)$. 
By equicontinuity, when $n$ is large enough and $x\in \D^k$ arbitrary, $\phi(x)$ maps every interval $[\frac{k-1}{n},\frac{k+1}{n}]$  into $B_R(\phi(x,\frac{k}{n}))\subset P$, for $k = 1,\dots,n-1$. 

Let $\phi_1 : \D^k \to \Omega_n$ such that $\phi_1(x)$ is the geodesic polygon with vertices at $\phi(x)(\frac{k}{n})$ for $k=0,\dots,n$. 
Using the unique shortest geodesic between $\phi(x)(s)$ and $\phi_1(x)(s)$ for each $s\in [0,1]$, we define a homotopy $\phi_t$ between $\phi$ and $\phi_1$ with $\phi_t(x) = \phi(x)$ when $x \in \partial \D^k$. 
\end{proof} 

For $\omega\in\Omega_n$ let  $\omega_k = \omega|{[\frac{k-1}{n},\frac{k}{n}]}$  for $k=1,\dots,n$. 
Its length is $L(\omega_k) = \frac{1}{n}|\omega_k'|$,  its energy $E(\omega_k) =  \frac{1}{n}|\omega_k'|^2 = n\.L(\omega_k)^2$.
The distance between the vertices $\omega(\frac{k-1}{n}),\omega(\frac{k}{n})$  is the length of $\omega_k$ which is $\leq\sqrt{E(\omega)/n}$ since {$E(\omega) \geq E(\omega_k) = n \cdot L(\omega_k)^{2} $}.

For $c > 0$ let $\Omega_n^c = \{\omega\in\Omega_n: E(\omega)\leq c\}$.
We have $\Omega_*^0 = \Omega_n^{c_o}$ where $c_o$ is the energy of a shortest geodesic.
By continuity, $E\o\phi_1$ has bounded image, hence $\phi_1(X)\subset \Omega_n^c$ for some $c>0$. 
When $n$ is large enough, more precisely $n>c/R^2$, any two neighboring vertices of every  $\omega \in\Omega_n^c$ lie in a common convex ball, hence the joining  shortest geodesic segments are unique and depend smoothly on the vertices.

Thus we may consider $\Omega_n^c$ as the closure of an open subset of $P\x\cdots\x P$ ($(n-1)$-times) with its induced topology.

\begin{lem} \label{dconn} 
The inclusion $\Omega^0_* \subset \Omega_n^c$ is $d$-connected for $d$ as in Theorem \ref{short}. 
\end{lem}

\proof
Let $\phi : \D^k \to \Omega_n^c$, $0 \leq k \leq d$,  with $\phi(\d \D^k) \subset \Omega_*^0$.
The space $\Omega_n^c$ is finite dimensional and contains all geodesics  of length $\leq\sqrt c$ from $o$ to $p$. 
It is a closed subset of $P\x \cdots\x P$ ({$(n-1)$}-times), and its boundary points are the polygons in $\Omega_n$ whose energy takes its maximal value $c$. 
The gradient of $-E$ on $\Omega_n^c$ is a smooth vector field and its flow is smooth, too. 

The index of any geodesic is the same in $\Omega_*$ and $\Omega_n$, see \cite[Lemma 15.4]{M}. 
Hence the index of a non-minimal geodesic $\tilde\gamma$ in $\Omega_n$ (a ``saddle'' for $E$) is at least $d+1$,  and $\phi_n(\D^k)$ can avoid the domains of attraction for all these  geodesics. 
The energy decreases along the gradient lines starting on $\phi_n(\D^k)$, thus these curves avoid the boundary of $\Omega_n^c$ and they end up on the minimum set of $E$,  the set of minimal geodesics. 
Hence we can use this flow to deform  $\phi$ into some $\tilde\phi:\D^k\to\Omega_* ^0$ without changing $\phi|_{\d\D^k}$.
\endproof

\noindent
\hskip 4cm
\includegraphics{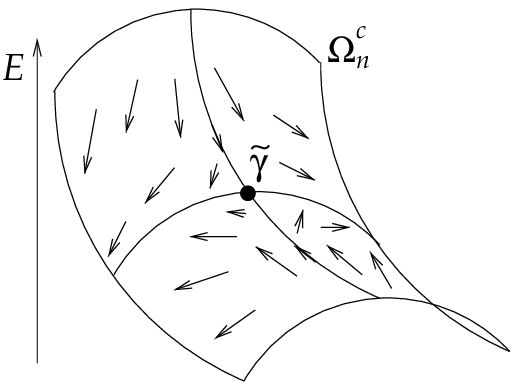} 

\begin{proof}[Proof of Theorem \ref{short}:] Let $\phi : \D^k \to \Omega_*$, $k \leq d$, with $\phi(\partial \D^k) \subset \Omega_*^0$.  
By Lemma \ref{72} we may (after a deformation which is constant on $\partial \D^k)$ assume  that $\phi(\D^k)\subset \Omega_n^c$ for some large $n$. 
Now the claim follows from Lemma \ref{dconn}. 
\end{proof}

\section{A lower bound for the index} \label{index}

How large is $d$ in Theorem \ref{short}? 
This has been computed in \cite[\S 23,24]{M} and \cite{EHa}.\footnote
	{A different argument using root systems was given by Bott \cite[6.7]{B} and in more
	detail by Mitchell \cite{Mt1,Mt2}} 
We slightly simplify Milnor's arguments replacing curvature computations by totally geodesic spheres. 
An easy example is the sphere itself, $P = \SS^n$. 
A non-minimal geodesic $\gamma$ between poles $o$ and $p$  covers a great circle at least one and a half times and can be shortened within any 2-sphere in which it lies (see figure below). 
There are $n-1$ such 2-spheres perpendicular to each other since the tangent vector  $\gamma'(0) = e_1$ (say) is contained in  $n-1$ perpendicular planes in the tangent space, 
$\SSpan(e_1,e_i)$  with $i\geq2$.
Thus the index is $\geq n-1$, in fact $\geq 2(n-1)$ since any such geodesic contains at least 2 conjugate points where it can be shortened by cutting the corner, see figure.

\ms\hskip3cm
\includegraphics{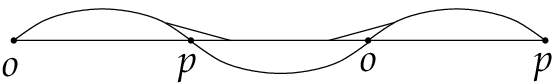}

\ms

For the groups $P = G = \SO_p,\U_p,\Sp_p$ ($p$ even) and their iterated centrioles we have
similar results.
The Riemannian metric on $G$ is induced from the inclusion $G\subset \End(\K^p)$
with $\K = \R,\C,\H$ respectively, where the inner product on $\K^{p\x p} = \End(\K^p)$ is
\beq \label{tracemet}
	\<A,B\> = \frac{1}{p}\Re\trace A^*B
\eeq
for any $A,B\in \K^{p\x p}$ where $A^* := \bar A^T$. In particular, $\<I,I\> = 1$.

\begin{prop} \label{SO-index}
The index of any non-minimal geodesic from $I$ to $-I$ in $\SO_p$   is at least $p-2$.
\end{prop}

\proof
A shortest geodesic from $I$ to $-I$ in $\SO_p$ is a product of $p/2$ half turns, planar rotations by the angle $\pi$ in $p/2$  perpendicular 2-planes in $\R^{p}$. 
A non-minimal geodesic must make  an additional full turn and thus a $3\pi$-rotation in at least one  of these planes, say in the $x_1x_2$-plane.
We project $\gamma$ onto a geodesic $\gamma_1$ in a subgroup $\SO_4 \subset \SO_{p}$ sitting in the coordinates $x_1,x_2,x_{k-1},x_{k}$, for any even $k\in\{4,\dots,p\}$. 
Then $\gamma_1$ consists of 3 half turns in the $x_1x_2$-plane together with (at least) one half turn in the  $x_{k-1}x_k$-plane. 

In the torus Lie algebra $\t$ of $\so_4$, which is 
$\R^2\cong \t =\left\{\bsm aJ \cr & bJ\esm: a,b\in\R\right\}$ for $J= \bsm & -1 \cr 1\esm$,
we have $\gamma_1(t) = \exp \pi t v$ with $v = \bsm 3J\cr & J\esm =\bsm 3 \cr 1\esm$ for short. 
Now we use $\SO_4 = L(\SS^3)R(\SS^3) \cong (\SS^3\x\SS^3)/\pm$, where the two $\SS^3$-factors correspond to the vectors $\bsm 1\cr 1\esm,\bsm 1\cr -1\esm\in\t$ 
(the left and right multiplications by $i\in\H$ on $\H=\R^4$). 
Decomposing $v$ with respect to this basis we obtain  $v = 2\bsm 1 \cr 1\esm + \bsm 1 \cr -1\esm$.
Thus the lift of $\gamma_1$ in $\SS^3\x\SS^3$ has two components: one is striding across a full great circle, the other across a half great circle. 
The first component passes  the south pole of $\SS^3$ and takes up index 2.
Since there are $(p-2)/2$ such coordinates $x_k$ the index of a non-minimal geodesic in $\SO_{p}$ is at least $p-2$ (compare \cite[Lemma 24.2]{M}).
\endproof

For $\U_p$ we have a different situation.  
To any path $\omega: I\leadsto -I$ we assign the closed curve $\det\omega : [0,1]\to \SS^1\subset\C$ (assuming $p$ to be even).
Its winding number (mapping degree) $w(\omega)\in\Z$ is obviously constant on each connected components of $\Omega(\U_p; I, -I)$.

\begin{defn} \label{wind}
We will call $w(\omega)$ the {\em winding number of $\omega$}.
\end{defn}

Any geodesic in $\U_p$ from $I$ to $-I$ is conjugate to $\gamma(t) = \exp(t\pi i D)$ where $D = \diag(k_1,\dots,k_p)$ with odd integers $k_i$, and $w(\gamma) = \2\sum_j k_j$. 
Let $k_j,k_h$ be a pair of entries with  $k_j>k_h$ and consider the  projection $\gamma_{jh}$ of $\gamma$ onto $\U_2$ acting on $\C e_j+\C e_h$.
Let $a = \2(k_j+k_h)$ and $b = \2(k_j-k_h)$. 
Then  
\[
   e^{-\pi i at}\gamma_{ij}(t)=\diag(e^{\pi i bt},e^{-\pi i bt}) \, . 
 \]
This is a geodesic in $\SU_2 \cong\SS^3$ which takes the value $\pm I$ when $t$ is a multiple of $1/b$. 
For $0 < t < 1$ there are $b-1$ such $t$-values, and at any of these points $\gamma_{ij}$ takes up index $2$. Thus
\beq \label{indUp}
	\ind(\gamma) \geq 2\sum_{k_j>k_h} \left(\2(k_j-k_h)-1\right).  
\eeq

\begin{prop} \label{U-index}
Let $\gamma : I\leadsto-I$ be a geodesic in $\U_p$ with $p-2|w(\gamma) |\geq 2d$ for some $d > 0$. 
Then $\ind(\gamma)> d$ unless $\gamma$ is minimal. 
\end{prop}

\proof 
Let $k_+$ be the sum of the positive $k_j$ and $-k_-$ the sum of the negative $k_j$. 
Then $k_+ - k_- = \sum_j k_j = 2w$ and $k_+ + k_-\geq p$.
Thus $2k_+\geq p+2w\geq p-2|w|$ and $2k_-\geq p-2w\geq p-2|w|$, and the assumption $p- 2|w| \geq 2d$  implies $k_\pm \geq d$.
If $k_j\geq 3$ or $k_h\leq -3$ for some $j,h$, then by \eqref{indUp} 
\[
     \ind(\gamma) \geq \sum_{k_h<0} (3-k_h-2) = 1+k_- > d 
\]
or 
\[
   \ind(\gamma) \geq \sum_{k_j>0} (k_j+3-2) = 1+k_+ > d \, . 
\]
At least one of these inequalities must hold, unless all $k_i\in\{1,-1\}$. 
In the latter case $\gamma$ is minimal: it has length $\pi$ which is the distance between $I$ and $-I$ in $\U_p$, being the minimal norm (with respect to the inner product \eqref{tracemet}) for elements of $(\exp|_\t)^{-1}(-I) = \{\pi i \diag(k_1,\dots,k_p): k_1,\dots,k_p \textrm{ odd}\}\subset\t$ where $\t$ is the Lie algebra of the maximal torus of diagonal matrices in $\U_p$.
\endproof

\medskip
Next we will show that similar estimates hold for arbitrary iterated centrioles $P_\ell$ of $\SO_p$. 
A geodesic $\gamma : J_\ell \leadsto -J_\ell$ in $P_\ell$ has the form
\beq \label{gamma} 
    \gamma(t) = e^{\pi t A}J_\ell \, , \quad t\in [0,1] \, , 
\eeq
for some skew symmetric matrix $A$, which commutes with $J_1,\dots,J_{\ell-1}$ and anticommutes with $J_\ell$, cf.\ \eqref{anticom}.
We split $\R^p$ as a direct sum of subspaces, 
\beq \label{Mj}
	\R^p = M_1\oplus\dots\oplus M_r
\eeq 
such that the subspaces $M_j$ are invariant under $J_1,\dots,J_\ell$ and $A$ 
and minimal with this property.
Then $A$ has only one pair of eigenvalues $\pm \i k_j$ on every $M_j$; otherwise we could split $M_j$ further.
Since $\exp(\pi A) = -I$, the $k_j$ are odd integers which can be chosen positive.

When all $k_j = 1$, then  $\gamma$ has minimal energy:   
$E(\gamma)=|\gamma'|^2 = |\pi A\gamma|^2 \buildrel \eqref{tracemet}\over = \pi^2$.  
In general, $J':=A/k_j$ is a complex structure on $M_j$, and $J_{\ell+1} := J_\ell J'$ is another complex structure on $M_j$ which anticommutes with $J_1,\dots,J_\ell$.
Hence each $M_j$ is an irreducible $\Cl_{\ell+1}$-module. 

The two cases (a) $\ell \neq 4m-2$ and (b) $\ell=4m-2$ have to be distinguished (cf. \cite[p.\ 144-148]{M}); they are similar to the previous cases of $\SO_p$ and $\U_p$, respectively. 
By (\ref{Mk}), all $M_j$ are isomorphic as $\Cl_\ell$-modules when $\ell=4m-2$ (Case (b)) and as $\Cl_{\ell+1}$-modules when $\ell \neq 4m-2$ (Case (a)).

\begin{prop} \label{neq4m-2} 
Suppose  $\underline{\ell\neq 4m-2}$. 
Let $s_{\ell+1}$ be the dimension of the irreducible $\Cl_{\ell+1}$-module $S_{\ell+1}$, see \eqref{Mk}. Then $s_{\ell+1}|p$\,, and all non-minimal geodesics $\gamma:J_\ell\leadsto-J_\ell$ in $P_\ell$ have index $> r-1$ with $r = p/s_{\ell+1}$.
\end{prop}

\proof
Choose any pair of submodules $M_j,M_h$ with $j\neq h$ in \eqref{Mj}. 
We may isometrically identify $M_j$ and $M_h$ as $\Cl_{\ell+1}$-modules, but first we modify the module structure on $M_h$ by changing the sign of $J_{\ell+1}$. 
In this way we view $M_j + M_h = M_j\oplus M_j$. Thus 
$$
	A|_{M_j+M_h} = \bpm k_j J' \cr & -k_h J'\epm = \bpm aJ'\cr & aJ'\epm + \bpm bJ'\cr & -bJ' \epm 
$$
for $a = \2(k_j-k_h)$, $b = \2(k_j+k_h)$. Let $\gamma(t)_{jh} := \gamma(t)|_{M_j+M_h}$. Then
\beq \label{etA}
	\gamma(t)_{jh} = (e^{t\pi A}J_\ell)_{jh} = 
	\bpm e^{t\pi aJ'} \cr & e^{t\pi aJ'}\epm\bpm e^{t\pi bJ'} \cr & e^{-t\pi bJ'}\epm
	\bpm J_\ell \cr & J_\ell\epm
\eeq
Put $B := \bsm &-I\cr I \esm$ on $M_j + M_h = M_j\oplus M_j$ and $B = 0$ on the other modules $M_k$ for $k\neq j,h$.
Then $B$ and $e^{uB}$ commute with $J_1,\dots,J_\ell$ for all $u\in\R$. 
Put $A_u = e^{uB}A e^{-uB}$. Then all geodesics 
$$
	\gamma_u(t)  = e^{uB}\gamma(t)e^{-uB} = e^{t\pi A_u}J_\ell
$$ 
are contained in $P_\ell$ by \eqref{anticom}.

The point $\gamma(t)$  is fixed under conjugation with the rotation matrix $e^{uB} = \bsm cI & -sI \cr sI & cI \esm$ on $M_j\oplus M_j$ with $c =\cos u$, $s=\sin u$  if and only if $e^{t\pi bJ'} = e^{-t\pi bJ'}$, see \eqref{etA}.
This happens precisely when $t$ is an integer multiple of $1/b$. 
If $k_h>1$, say $k_h \geq 3$, then $b =\2(k_j+k_h) \geq 2$ and $1/b \in (0,1)$.
All $\gamma_u$ are geodesics in $P_\ell$ connecting $J_\ell$ to $-J_\ell$. 
Using ``cutting the corner'' it follows that $\gamma$ can no longer be locally shortest beyond $t\!=\! 1/b$ , see figure below. 
If there is at least one eigenvalue $k_h > 1$, there are at least $r-1$ index pairs $(j,h)$ with $\2(k_j+k_h)\geq 2$,  hence by \eqref{indUp}, the index of non-minimal geodesics is at least $r-1$.\footnote
	{Another way of saying: $J_\ell$, $J'$ and $BJ'$ span a Clifford 2-sphere in $\SO(M_j+M_h)$
	anticommuting with $J_1,\dots,J_{\ell-1}$ and containing $\gamma_{jh}$ which covers at least three
	half great circles.}

\ms
\hskip 3.5cm
\includegraphics{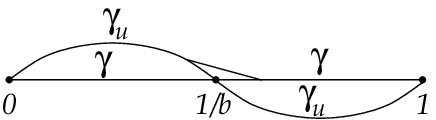}
\endproof

\bigskip\medskip
Now we consider the case  $\underline{\ell = 4m-2}$.
Then $J_o := J_1J_2\dots J_{\ell-1}$ is a complex structure\footnote
	{\label{volume} 
	Let $J_1,\dots,J_{n}$ be a Clifford family and
	put $w_n = J_1\cdots J_n$. Then 
	$w_n^2 = J_1\cdots J_nJ_1\cdots J_n = (-1)^{n-1}w_{n-1}^2J_n^2 = (-1)^nw_{n-1}^2$.
	Thus $w_n^2 = (-1)^sI$ with $s = n+(n-1)+\dots+1 = \2 n(n+1)$. 
	When $n-1$ or $n-2$ are multiples of $4$ 
	then $s$ is odd, hence $w_n^2 = -I$. Further,
	when $n$ is even (odd), $w_n$ anticommutes (commutes) with $J_1,\dots,J_n$. 
	If the Clifford family extends to $J_1,\dots,J_{n+1}$, then $w_n$
	commutes with $J_{n+1}$ when $n$ is even and anticommutes when $n$ is odd.}
which commutes with $A$ and $J_1,\dots,J_{\ell-1}$  and anticommutes with $J_\ell$ (since $\ell-1$ is odd).  
Thus $A$ can be viewed as a complex matrix, using $J_o$ as the multiplication by $\i$,
and the eigenvalues of $A$ have the form $\i k$ for odd integers $k$. 
As before, we split $\R^p$ into minimal subspaces $M_j$ which are invariant under $J_1,\dots,J_\ell$ and $A$. 
Let $E_{k_j} \subset M_j$ be a complex  eigenspace of $A|_{M_j}$ corresponding to an eigenvalue $\i k_j$.
Then $E_{k_j}$ is invariant under $J_1,\dots,J_{\ell-1}$ (which commute with $A$ and $\i$), and also under $J_\ell$ (which anticommutes with both $A$ and $\i = J_o$). 
By minimality we have $M_j = E_{k_j}$, hence $A = k_j J_o$ on $M_j$. 

Again  we consider two such modules $M_j,M_h$. 
As $\Cl_\ell$-modules they can be identified, $M_j+M_h = M_j\oplus M_j$, see \eqref{Mk}. 
This time, 
\[
   A|_{M_j+M_h} = \diag(k_j\i,k_h\i) = a\i I+\diag(b\i,-b\i)
\]
with $a = \2(k_j+k_h)$ and $b = \2(k_j-k_h)$.
Thus 
\beq \label{gammat}
\gamma(t) = e^{\pi t A}J_\ell =  e^{\pi ta\i}\diag(e^{\pi t\i}, e^{-\pi tb\i}) J_\ell.
\eeq 
Consider the linear map $B$ on $\R^p$ which is $\bsm &-I\cr I\esm$ on $M_j+M_h = M_j\oplus M_j$ and $B = 0$ elsewhere, and the family of geodesics 
$$
	\gamma_u(t)  = e^{uB}\gamma(t)e^{-uB} = e^{\pi t A_u}J_\ell,\ \ A_u = e^{uB}A e^{-uB}.
$$ 
Since $e^{uB}$ commutes with $J_1,\dots,J_\ell$, the first equality 
implies $\gamma_u$ in $P_\ell$, see \eqref{anticom}. 
Now $\gamma(t)$ is fixed under conjugation with the rotation matrix $e^{uB} = \bsm cI & -sI \cr sI & cI \esm$ with $c=\cos u$, $s=\sin u$ if and only if $e^{\pi tb\i} = e^{-\pi tb\i}$ which happens precisely when $t$ is an integer multiple of $1/b$.
When $b>1$, we obtain an  energy-decreasing deformation by cutting $b\!-\!1$ corners, see figure above. Thus the index of a geodesic $\gamma$ as in \eqref{gamma} is similar to  \eqref{indUp}:
\beq \label{indb}
	\ind(\gamma) \geq \sum_{k_j>k_h} \left(\2(k_j-k_h)-1\right).
\eeq
As before we need a lower bound for this number  when $\gamma$ is non-minimal.

Any $J \in P_\ell$ defines a $\C$-linear map $JJ_\ell^{-1} = -JJ_\ell$ since $JJ_\ell$ commutes with all $J_i$ and hence with $J_o$. 
This gives an embedding\footnote
	{This embedding $J\mapsto JJ_\ell^{-1}$ is equivariant: Let $J = gJ_\ell g^{-1}$  
	for some $g\in \SO_p$ preserving
	$P_\ell$, then $JJ_\ell^{-1} = g\tau(g)^{-1}$ where $\tau$ is the involution
	$g\mapsto J_\ell g J_\ell^{-1}$. 
	It is totally geodesic, a fixed set component of the isometry $\iota\o\tau$ with 
	$\iota(g) = g^{-1}$ since 
	$\iota(\tau(g\tau(g)^{-1})) = (\tau(g)g^{-1})^{-1} = g\tau(g)^{-1}$.
	In fact it is the Cartan embedding of $P_\ell$ into the group 
	$G_\ell = \{g\in \SO_p: gP_\ell g^{-1} 	= P_\ell\}$ .}
\begin{equation} \label{embeddingPU} 
	P_\ell \hra \U_{p/2} : J\mapsto JJ_\ell^{-1}
\end{equation} 
Note that $p$ is divisible by four  due to the representation theory of Clifford algebras \eqref{Mk}, since for $\ell = 4m-2$ the space $\R^p$ admits at least two anticommuting almost complex structures. 

\begin{defn} \label{w} 
For any $\omega \in \Omega:= \Omega(P_\ell; J_\ell,-J_\ell)$ we define its {\em winding number} 
$w(\omega)$ as the winding number of  $\omega J_\ell^{-1}$, considered as a path  in $\U_{p/2}$
from $I$ to $-I$, see Def.\ {\rm \ref{wind}}. In particular for a geodesic $\omega = \gamma$
with $\gamma(t) = e^{\pi t A}$ as in \eqref{gammat} we have $w = \2\dim_{\C}(S_\ell) \sum_j k_j$.
\end{defn}

\begin{prop} \label{4m-2}
Let $\ell = 4m-2$. 
Let $\gamma:J_\ell\leadsto-J_\ell$ be a geodesic in $P_\ell$ with winding number $w$ such that
\beq \label{p-2w}
	p-4|w|\geq 4d\.s_\ell
\eeq
for some $d > 0$ where $s_\ell = \dim_{\R}  S_\ell$, see \eqref{Mk}. Then $\ind(\gamma) > d$ unless $\gamma$ is minimal.
\end{prop}

\proof 
Let $r$ be the number of irreducible $\Cl_\ell$-representations in $\R^p$,
that is $r = p/s_\ell$.
Let $k_+$ be the sum of the positive $k_j$ and $-k_-$ the sum of negative $k_j$. 
Then $k_+ + k_-\geq r = p/s_\ell$ and $k_+ - k_- = \sum_j k_j = 2w/\dim_{\C}S_\ell = 4w/s_\ell$.
Thus $2k_+\geq (p+4w)/s_\ell$ and $2k_-\geq (p-4w)/s_\ell$,
and our assumption $(p\pm 4w)/s_\ell \geq 4d$ implies $k_\pm \geq 2d$. 
If $k_j\geq 3$ or $k_h\leq -3$ for some $j,h$, then by \eqref{indb}
\[
     \ind(\gamma) \geq \sum_{k_h<0} \2(3-k_h-2) = \2(1+k_-) > d 
\]
or 
\[
   \ind(\gamma) \geq \sum_{k_j>0} \2(k_j+3-2) = \2(1+k_+) > d \, . 
\]
One of these inequalities holds unless all $k_i\in\{1,-1\}$ which means that $\gamma$ is minimal, see the subsequent remark.
\endproof

\begin{rem} \label{kpm}
Let still $\ell = 4m-2$.
The geodesic $\gamma_o(t) = e^{\pi t A }$ from $I$ to $-I$ is minimal in $U_{p/2}$ (with length $\pi$) if and only if $A$ has only eigenvalues $\pm\i$. 
In this case, when $\gamma=\gamma_oJ_\ell$  
lies inside $P_\ell$, the numbers $k_+$ of positive signs and $k_-$ of negative signs are determined by the above conditions $k_+ + k_- = p/s_\ell$ and $k_+ -k_-  = 4w/s_\ell$.
This corresponds to a $(J_1,\dots,J_\ell)$-invariant orthogonal splitting  (which in particular is complex linear with respect to $\i=J_1\cdots J_{\ell-1}$), 
\beq \label{V0V1}
	\R^p = \C^{p/2} = L_- \oplus L_+ \,,\ \ \dim_{\C} L_\pm  = k_\pm\,, 
\eeq 
with $A = \pm \i$ on $L_\pm$. 
Then $A^2 = -I$, and we obtain another complex structure $J_{\ell+1}  = AJ_\ell =\gamma(\2)$  anticommuting with $J_1,\dots,J_\ell$ and defining a minimal centriole $P_{\ell+1}\subset P_\ell$.
Note that 
\[
    J_{\ell+1} = \pm J_1 \cdots J_{\ell} \text{ on } L_\pm \,. 
\]
The property of $J_{\ell+1}$ being determined (up to sign) by $J_1, \ldots, J_\ell$ only appears for $\ell = 4m-2$.

\end{rem}

For later use we formulate a necessary and sufficient condition when this $\Cl_{\ell+1}$-representation can be extended to a $\Cl_{\ell+2}$-representation. 

\begin{prop} \label{Pell+2}
Let $\gamma(t) = e^{\pi t A}J_\ell$ be a minimal geodesic in $\Omega(P_\ell ; J_\ell, - J_\ell)$ with mid-point $J_{\ell+1}$ for $\ell = 4m-2$. 
Then the following assertions are equivalent. 
\begin{enumerate}[(i)] 
 \item {There exists a complex structure $J_{\ell+2}$ anticommuting with $J_1, \ldots, J_{\ell+1}$}, 
 \item $w(\gamma)  = 0$.
 \end{enumerate} 
\end{prop}

\proof
With the notation from Remark \ref{kpm} recall  that $J_{\ell+1}J_\ell^{-1} = e^{(\pi/2)A} = A$ is a complex matrix with eigenvalue $\i$ on $L_+$ and $-\i$ on $L_-$. 

When $J_{\ell+2}$ exists, then $J_{\ell+2}J_\ell^{-1}$ is another complex matrix which anticommutes with $J_{\ell+1}J_\ell^{-1}$ and thus interchanges the eigenspaces $L_+$ and $L_-$.
Therefore $k_+ = k_-$, that is $w=0$. 

Vice versa, when $k_+=k_- =: k$, then $L_+$ and $L_-$ may be identified as $\Cl_\ell$-modules, and by putting $J_{\ell+2}:= \bsm & -I_k \cr I_k \esm\. J_\ell$ 
we define a further complex structure anticommuting with $J_1,\dots,J_{\ell+1}$.
\endproof

\begin{rem} \label{generalcase} 
We do not need to consider the group $G = \Sp_p$ and its iterated centrioles since $\Sp_p = P_4(\SO_{8p})$, cf.\ \eqref{table1}.
For $G = \U_p$, the iterated centrioles are $\U_q$ and  $\G_{q/2}(\C^q)$ (the complex Grassmannian) for all $q = p/2^m$. 
$\U_q$ has been considered in Prop.\ \ref{U-index}. 
In $\G_{q/2}(\C^q)$ the index of non-minimal geodesics is high when $q$ is large enough\footnote
	{This example was omitted in \cite{M}.} 
as can be seen from the {\em real Grassmannian} $\G_{q/2}(\R^q) = P_3(\Sp_q)$. 
In fact, real and complex Grassmannians have a common maximal torus, therefore any geodesic in the complex Grassmannian can be conjugated into the real Grassmannian.
Moreover, the distance between poles (i.e.\ the length of a minimal geodesic) agrees for the complex Grassmannian and the real one. Hence the index of a non-minimal geodesic in the complex Grassmannian is at least as big as in the real Grassmannian, but the latter one is among the  spaces considered in Prop.\ \ref{neq4m-2}.

\end{rem}

\section{Deformations of mapping spaces}  \label{defmapspace} 

We consider a round sphere $\SS=\SS^k$, $k\geq 1$, with ``north pole'' $N = e_{0}\in \SS$, and its equator $\SS' = \SS\cap N^\perp$ with ``north pole'' $N' = e_1\in \SS'$ (or rather ``west pole'', according to the orientation $(e_0,e_1)$), where the standard basis of $\R^{k+1}$ is denoted $e_0,\dots,e_k$. 
Let $\mu:[0,1]\to\SS$ be the meridian from $N$ to $-N$ through $N'$ and $m = \mu([0,1]) \subset \SS$. 

\ms
\hskip 5cm
\includegraphics{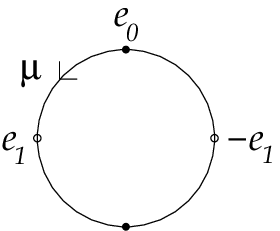}

\ms
Further let $P$ be the group $G \in \{\SO_p,\U_p,\Sp_p\}$ or one of its iterated centrioles  (see Section \ref{centrioles}). 
We fix some $J\in P$ and a shortest geodesic 
\[
     \gamma : J \leadsto -J
\]
in $P$, which we consider as a map $\gamma : m \to P$.

\begin{defn} \label{mapstar}
Let $\Map_*(\SS,P)$ be the space of $C^0$-maps $\phi:\SS\to P$ with $\phi(N) = J$, equipped with the compact-open topology.
We denote by $\Map_{\gamma}(\SS,P)$ the subset of maps $\phi$ with $\phi|_m = \gamma$, equipped with the subspace topology. 
\end{defn}

Building on the results from Section \ref{paths} will will develop a deformation theory for these mapping spaces $\Map_*(\SS,P)$. 
The main result, Theorem \ref{mindef}, says that stably it can be approximated by the subspace of maps which are block sums of constant maps and Hopf maps associated to Clifford representations (see Definition \ref{HopfCliff}).

\begin{lem} \label{1} 
The inclusion $\Map_\gamma (\SS, P) \subset \Map_*(\SS, P)$ is a weak homotopy equivalence. 
\end{lem} 

\begin{proof}  
This can be seen directly by constructing a deformation of arbitrary maps $\phi\in\Map_*(\SS, P)$ into maps $\tilde\phi\in\Map_\gamma (\SS, P)$. 
However, the proof becomes shorter if we use some elementary homotopy theory, see \cite[Ch.\ VII]{Br}. 
We consider the restriction map
\[
      r :   \Map_*(\SS, P) \to \Map_*(m, P) \, , \quad \phi \mapsto \phi|_m , 
\]
with $\Map_*(m, P) := \{f:m\to P: f(N) = J\}$. 
Clearly, $r^{-1}(\gamma) = \Map_\gamma (\SS, P)$.
This map is a fibration, according to Theorem 6.13 together with Corollary 6.16 in \cite[Ch.\ VII]{Br}, applied to  $Y = P$ and $(X,A) = (\SS,m)$; the requirement that $A\subset X$ is a cofibration follows from Corollary 1.4 in  \cite[Ch.\ VII]{Br}. 
Since the base space is clearly contractible, the inclusion of the fibre
\beq \label{Mapgamma}
	\Map_\gamma (\SS, P)  \subset \Map_*(\SS, P)
\eeq
is a weak homotopy equivalence, due to the homotopy sequence of this fibration.
\end{proof}

Let $\SS'=\SS^{k-1}\subset\SS$ be the equator sphere. 
We may parametrize $\SS$ by $\SS'\x I$ with $(v,t)\mapsto \mu_v(t)$ where $\mu_v : [0,1]\to \SS$ is the meridian from $N$ to $-N$ through $v\in\SS'$. 
Using the assignment $\phi\mapsto (v\mapsto \phi\o\mu_v)$ for any $\phi\in\Map_\gamma(\SS,P)$ and $v\in\SS'$  we obtain a canonical homeomorphism  
\beq \label{S'S}
	\Map_\gamma(\SS,P) \approx \Map_*(\SS',\Omega) 
\eeq
where 
\[
   \Omega = \Omega(P ; J, - J) 
\]
is the space of continuous paths  $\omega : [0,1]\to P$ with $\omega(0) = J$ and $\omega(1) = -J$, and $\Map_*(\SS',\Omega)$ the space of mappings $\phi :\SS'\to\Omega$ with $\phi(N') = \gamma$.

Let $\Omega^0= \Omega^0(P; J , -J)  \subset \Omega$ be the subspace of minimal geodesics from $J$ to $-J$ in $P$. 
Since every such minimal geodesic  is determined by its midpoint, which belongs to the midpoint set $\hat P'\subset P$ whose components are the minimal centrioles $P' \subset P$ (see Section \ref{centrioles}), we furthermore have a canonical homeomorphism 
\beq \label{Omega0}
    \Map_*(\SS',\Omega^0) \approx  \Map_*(\SS',\hat P')    \, . 
\eeq
The composition 
\bea \label{SigmaJ}
    \Sigma_J\ : \ \Map_*(\SS',\hat P') &\approx& \Map_*(\SS',\Omega^0) \subset \Map_*(\SS',\Omega) \\
  &\approx& \Map_\gamma(\SS,P) \ \subset \Map_*(\SS,P) \nonumber
\eea
sends $\phi'\in\Map_*(\SS',\hat P')$ to $\phi \in \Map_*(\SS,P)$ defined by 
\beq \label{gammaphi}
	\phi(v,t) = \gamma_{\phi'(v)}(t)\ \ \textrm{ for all } (v,t) \in \SS'\x[0,1] 
\eeq
where $\gamma_{\phi'(v)} : J \leadsto -J$ is the shortest geodesic in $P$  through  $\phi'(v)$.

\begin{defn} The map $\Sigma_J$  is called the {\em geodesic suspension} along $J$. 
\end{defn} 

We will show that $\Sigma_J$ is highly connected in many cases. 
At first we will deal with the case $k=\dim \SS \geq 2$.

\begin{prop} \label{corl}
Let $k \geq 2$ and let $d+\dim( \SS')$ be smaller than the index of any non-minimal geodesic in the connected component $\Omega_* \subset \Omega$ containing $\gamma$. 
Then $\Sigma_J$ is $d$-connected. 
\end{prop} 

\begin{proof} 
By Lemma \ref{1}, the inclusion $ \Map_\gamma(\SS,P) \to \Map_*(\SS,P)$ is a weak homotopy equivalence. 
Therefore we only have to deal with the  inclusion 
\[
    \Map_*(\SS',\Omega^0) \subset \Map_*(\SS',\Omega) \, . 
\]
Let $\psi : \D^j \to\Map_*(\SS',\Omega)$, $j \leq d$, with  $\psi(\d \D^j )\subset \Map_*(\SS',\Omega^0)$. 

We consider $\psi$ as a map $ \hat\psi :  \D^j \x\SS' \to\Omega$ and observe, using $k \geq 2$,  that the image  of this map lies in the connected component $\Omega_* \subset \Omega$ determined by  $\gamma$. 

Theorem \ref{short} may now be applied to the  $(j+ \dim \SS')$-dimensional CW-pair (compare Definition \ref{dconnected})
\[
    (X,Y) = (\D^j  \times \SS', (\partial \D^j \times \SS') \cup (\D^{j} \times \{N'\})) 
\]
where $j + \dim \SS' \leq d + \dim \SS'$. 
This results in a deformation of $\hat\psi$  to a map with image contained in $\Omega^0_*$ by a deformation which is constant on $(\partial \D^j \times \SS') \cup (\D^j \times \{N'\}) $. 

Hence $\psi$ may be deformed to a map with image contained in $\Map_*(\SS', \Omega^0)$ by a deformation which is constant on $\partial \D^j$. 
\end{proof} 

\noindent
The following picture  
illustrates the deformation process employed in this proof where $\varphi\in\Map_\gamma(\SS,P)$ lies in the range of $\psi$.

\smallskip
\hskip 1cm
\includegraphics{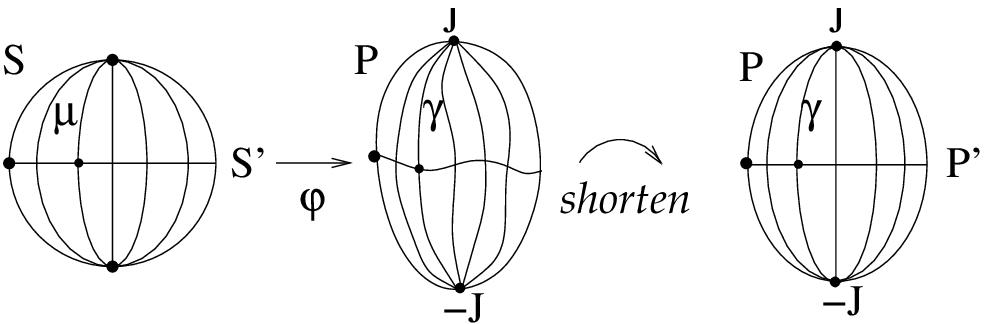}

\section{Iterated suspensions and Clifford representations}

This process can be iterated.
For the sake of exposition we will concentrate on the case $G = \SO_p$ in the following argument. 
Let $k \geq 1$ and consider an orthogonal $\Cl_k$-representation 
\[
     \rho: \Cl_k \to \R^{p\x p}  
\]
 determined by  anticommuting complex structures $J_1, \ldots, J_k \in \SO_p$. 
Let 
\beq \label{chain}
    \SO_p = P_0  \supset P_1 \supset \cdots \supset P_k
\eeq
be the associated chain of minimal centrioles (compare Theorem \ref{thmcentrioles}). 

Let $e_0,\dots,e_k$ denote the standard basis of $\R^{k+1}$ and $\SS^k \subset \R^{k+1}$ be the unit sphere. Furthermore, for $ \ell \in \{0 , \ldots, k\}$ let\footnote  
	{The reason for this unusual choice is that \eqref{chain} is descendent
	and we would like to map $e_1$ onto $J_1\in P_1$ and $e_j$ onto $J_j\in P_j$
	as in Theorem \ref{thmcentrioles}.}
$$
	\SS^{k-\ell} := \SS^k\cap \SSpan\{e_\ell,\dots,e_k\}
$$ 
be the great sphere in the subspace spanned by $e_\ell,\dots,e_k$ with ``north pole'' (base point) $e_{\ell}$ and equator sphere $\SS^{k-\ell -1}$.

Let $\Map_*(\SS^{k-\ell} , P_{\ell})$, $\ell= 0, \ldots, k$, be the space of maps $\SS^{k-\ell}  \to P_{\ell}$ sending $e_{\ell}$ to $J_{\ell}$  for $\ell \geq 1$ and sending $e_0$ to $I\in \SO_p$ for $\ell = 0$. 
Putting $\Sigma_\ell := \Sigma_{J_\ell}$ (the suspension along $J_\ell$) with $J_0 := I$ we consider the composition 
\begin{equation} \label{theta} 
 \theta:   \Map_*(\SS^1, P_{k-1})  \stackrel{\Sigma_{{k-2}}}{\longrightarrow} \Map_*(\SS^2, P_{k-2})  \stackrel{\Sigma_{{k-3}}}{\longrightarrow} \cdots \stackrel{\Sigma_{{0}}}{\longrightarrow}  \Map_*(\SS^k  ,  \SO_p) \, . 
\end{equation} 

\begin{prop} \label{highcon} For each $d$  there is a $p_0=p_0(d)$ such that  $\theta$  is $d$-connected whenever $p \geq p_0$. 
\end{prop} 

\begin{proof}  Since the composition of $d$-connected maps is $d$-connected it suffices to show  that for all $\ell = 0 , \ldots, k-2$ the suspension
\[
     \Sigma_{J_{\ell} } :  \Map_*(\SS^{k-\ell-1}, P_{\ell+1} ) \to \Map_*(\SS^{k-\ell} , P_{\ell}) 
\]
 in the composition $\theta$ is $d$-connected once  $p \geq p_0$. 

This claim follows from Proposition \ref{corl}, where the assumption on  indices of non-minimal geodesics holds for the following reasons: 
\begin{enumerate}
   \item If $\ell+2$ is not divisible by four it holds by Proposition \ref{neq4m-2}. 
   \item If $\ell+2$ is divisible by four it holds by  Proposition \ref{4m-2}. 
     Indeed, let $\gamma : J_\ell \leadsto -J_\ell$ be the minimal geodesic in $P_\ell$ running through $J_{\ell+1}$. 
     Then $w(\gamma) = 0$ by Proposition \ref{Pell+2} since there is a complex structure $J_{\ell+2}$ anticommuting with $J_1, \ldots, J_{\ell+1}$ (recall $\ell+ 2 \leq k$). 
     Hence $w(\tilde \gamma)=0$ for each non-minimal geodesic $\tilde \gamma$ in the connected component $\Omega_* \subset \Omega = \Omega(P_\ell; J_\ell , - J_\ell)$ containing $\gamma$. 
\end{enumerate} 
\end{proof} 

It remains to investigate the space $\Map_*(\SS^1, P_{k-1})$. 
We restrict to the case $k= 4m-1$. 
Setting $\ell = k - 1$ this means  $\ell = 4m-2$, and $P = P_{\ell}$. 
From equation \eqref{embeddingPU} we recall the canonical embedding 
\[
    P_\ell \hookrightarrow \U_{p/2} \, , \quad J \mapsto JJ_\ell^{-1} \, . 
\]

\begin{assum} In the remainder of this section we will assume that the given $\Cl_k$-representation $\rho$ satisfies 
\[
   J_k = + J_1 \cdots J_{\ell} \, .
\]
Orthogonal $\Cl_k$-representations of this kind and their Hopf maps will be called 
{\em positive}.\footnote 
	{\label{1234567}
	For $k=3$ and $7$ this is the representation of $Cl_k$ on $\K \in\{\H,\O\}$
	by left translations. In fact, on $\H$ we have $ijk = -1$ (Hamilton's equation).
	On $\O$ there are further basis elements $\ell$, $p = i\ell$, $q = j\ell$, $r = k\ell$,
	and the corresponding left translations on $\O$ will be denoted by the same symbols.
	Then $ijk = -\id$ on $\H$.
	Further, using anti-associativity of Cayley triples like $(j\ell)r = -j(\ell r)$
	we have $\ell pqr = \ell(p((j\ell)r)) = -\ell(p(j(\ell r))) = -\ell(p(j(\ell k\ell))) 
	=-\ell(p(jk)) = -\ell((i\ell)i) = -\ell^2 = 1$ and thus $\ell pqr = \id$ on $\H$ 
	since $\ell pqr$ commutes with $i,j,k$. Moreover, both $ijk$ and $\ell pqr$ 
	anticommute with $\ell$, hence on $\ell\H$ we have $ijk = \id$ and $\ell pqr = -\id$.
	Together we see $ijk\ell pqr = -\id$ on $\O$.}

Furthermore we denote by 
\[
    \gamma : J_\ell \leadsto - J_\ell
\]
 the shortest geodesic in $P_\ell$  through $J_{\ell+1} = J_k$. 
With the complex structure $\i = J_o = J_1 J_2 \cdots J_{\ell-1}$ on $\R^p$ we then have 
\[
    \gamma(t) = e^{\pi \i t} J_\ell  
\]
In particular {$w(\gamma) =  p/4$}. 
\end{assum}

\begin{defn} \label{windingloop} 
For each $\omega \in \Map_*(\SS^1, P_\ell)$ with $\ell = k-1$ 
we define the {\em winding number} $\eta(\omega) \in \Z$ as the winding number of the composition 
\[
	\SS^1 = \SS^{k-\ell} \stackrel{\omega}{\longrightarrow}  P_\ell \hookrightarrow \U_{p/2} \stackrel{\det}{\longrightarrow} \SS^1 \, .
 \]
The  first map identifies $(\alpha,\beta) \in \SS^1 \subset \R^2$ and  $\alpha e_{k-1} + \beta e_k \in \SS^{k-\ell} \subset {\rm Span} \{ e_{k-1}, e_k\}$. Note that $\eta$ is constant on path components of $ \Map_*(\SS^1, P_\ell)$.

For $\eta \in \Z$ let 
\[
    \Map_*(\SS^1, P_{\ell})_{\eta} \subset  \Map_*(\SS^1, P_{\ell})
\]
denote the subspace of loops with winding number equal to $\eta$. 
This is a union of path components of $\Map_*(\SS^1, P_{\ell})$. 

By Proposition \ref{highcon}, for sufficiently large $p$, the map $\theta$ in \eqref{theta} induces a bijection of path components.
Hence, at least after taking a block sum with a constant map $\SS^k \to \SO_{p'}$ for some large $p'$, the previously defined winding number
\[
  \eta :  \Map_*(\SS^1, P_{\ell}) \to \Z 
\]
induces a map 
\[
   \eta^s :  \Map_*(\SS^k, \SO_p) \to \Z \, , 
\]
such that $\eta^s\o\theta = \eta$, which we call the {\em stable winding number}.

This is independent of the particular choice of $p'$, and hence well defined on $\Map_*(\SS^k , \SO_p)$ for the original $p$. 
By definition it is constant on path components and therefore may be regarded as a map 
\[
    \eta^s : \pi_k ( \SO_p; I) \to \Z \, . 
\]
Since all positive $\Cl_k$-representations on $\R^p$ are isomorphic, this map is independent from the chosen positive $\Cl_k$-representation $\rho$. 

For $\eta \in \Z$ let 
\[
    \Map_*(\SS^k, \SO_p)_\eta \subset \Map_*(\SS^k , \SO_p) 
\]
denote  the subspace with  stable winding number equal to $\eta$. 
\end{defn} 

\begin{ex} \label{winding_for_hopf} 
Let $\mu : \SS^k \to \SO_p$ be the Hopf map (see Definition \ref{HopfCliff}) associated to the given Clifford representation $\rho$. 
Then $\eta^s(\mu) = p/2$. 
Indeed we have $\mu = \theta ( \omega)$, where $\omega : \SS^1 \to P_{\ell}$ is given by $\alpha \cdot e_{k-1} + \beta \cdot e_{k} \mapsto  \alpha \cdot J_{k-1} + \beta \cdot J_{k}$, such that with $\i = J_1 J_2 \cdots J_{\ell-1}$ the composition $\SS^1 \to P_\ell \to \U_{p/2}$ is given by $e^{\pi i t} \mapsto e^{ \pi \i t} J_{\ell}$ as $J_k = + J_1 \cdots J_\ell = \i J_\ell$. 
\end{ex}

\begin{rem} \label{windingadditive} The stable winding number is additive with respect to block sums of $\Cl_k$-representations. 
More precisely, let 
\[
    \rho_i : \Cl_k \to \SO_{p_i} 
\]
be positive $\Cl_k$-representations, $i = 1,2$, with associated chains of minimal centrioles 
\[
    \SO_{p_i} \supset P^{(i)}_1 \supset \cdots \supset P^{(i)}_k \, . 
\]
The chain of minimal centrioles associated to the block sum action 
\[
    \rho = \rho_1 \oplus \rho_2 : \Cl_k \to \SO_{p_1}\x\SO_{p_2}\subset\SO_{p_1+p_2} 
\]
 then takes the form 
\[
    \SO_{p_1}\x\SO_{p_2}  \supset P^{(1)}_1\x P^{(2)}_1 \supset \cdots \supset P^{(1)}_k\x P^{(2)}_k \,, 
\]
and each suspension in \eqref{theta} 
is a product of corresponding suspensions for $\rho_1$ and $\rho_2$. 
Hence for $\phi = \phi_1 \oplus \phi_2 \in \Map_*(\SS^k, \SO_{p_1 + p_2})$ of block sum form we obtain for the stable winding numbers
\[
       \eta^s (\phi) = \eta^s(\phi_1) + \eta^s(\phi_2) \, . 
\]
\end{rem}

\bigskip
Let $\Omega = \Omega(P_\ell; J_\ell, - J_\ell)$ be the space of paths in $P_\ell$ from $J_\ell$ to $-J_\ell$ and let 
\[
   \hat P_k =  \Omega^0   \subset \Omega
\]
be the subspace of shortest geodesics, cf.{} \eqref{Pj}.
Consider the geodesic suspension 
\[
   \Sigma_{\ell} : \Map_*(\SS^0, \hat P_k) \to \Map_*(\SS^1, P_\ell) \, ,
\]
where $\Map_*(\SS^0, \Omega^0)$ is the space of maps $\SS^0 = \{ \pm e_k\} \to \Omega^0$ sending $e_k$ to $\gamma \in \Omega^0$.

\begin{defn} 
For $\phi  \in  \Map_*(\SS^0, \Omega)$ we denote by $w(\phi) \in \Z$ the winding number of $\tilde{\gamma} = \phi(-e_k) \in \Omega$ in the sense of Definition \ref{w}. 
Note that  possibly $w(\tilde \gamma) \neq w(\gamma)$, since $\SS^0$ is disconnected and hence $\tilde \gamma = \phi(-e_k)$ and $\gamma = \phi(e_k)$ may lie in different path components of $\Omega$. 

For $w \in \Z$ let 
\[
     \Map_*(\SS^0, \Omega)_w \subset  \Map_*(\SS^0, \Omega)  \text{ and }   \Map_*(\SS^0, \Omega^0)_{w} \subset  \Map_*(\SS^0, \Omega^0)
\]
be the subspaces of maps of winding number $w$. 

\ms
\hskip 5cm
\includegraphics{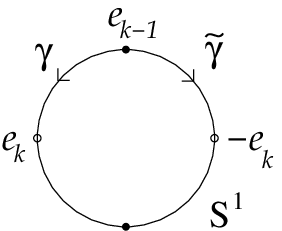}
\ms

Note that the concatenation map $\tilde\gamma \mapsto \gamma\ast\tilde\gamma^{-1}$ 
(first $\gamma$, then $\tilde\gamma^{-1}$ where $\tilde\gamma^{-1}(t) := \tilde\gamma(1-t)$, which together form a loop starting and ending at $J_{k-1}$) gives a canonical identification
\[
    \Map_*(\SS^0, \Omega)_w= \Map_*(\SS^1, P_\ell)_{\eta = (p/4) - w} 
\]
since $w(\gamma) =  p/4$.
\end{defn} 

After these preparations we obtain the following analogue of Proposition  \ref{corl}. 

\begin{prop} \label{corlk1} 
Let $d \geq 0$ and let $w \in \Z$ satisfy $ p - 4 | w | \geq 4 s_\ell\. d$. 
Then the suspension $\Sigma_{\ell}$ restricts to a $d$-connected map 
\[
   \Sigma_\ell : \Map_*(\SS^0, \Omega^0)_w= \Map_*(\SS^1, P_\ell)_{\eta = (p/4) - w} \, . 
\]
\end{prop} 

\begin{proof} Similar as in the proof of Proposition \ref{corl} we only have to deal with the inclusion
\[
   \Map_*(\SS', \hat P_k)_w \subset \Map_*(\SS', \Omega)_w \, 
\]
where we now have  $\SS' = \SS^0$, cf. \eqref{SigmaJ}. 
This is equivalent to the inclusion 
\[
   \Omega^0(P_\ell; J_\ell,-J_\ell)_w \subset \Omega(P_\ell; J_\ell,-J_\ell)_w\, . 
\]
By Theorem \ref{short} and Proposition \ref{4m-2} with our assumption on $|w|$ this inclusion is $d$-connected. 
\end{proof}

Each $\phi \in \Map_*(\SS^0, \Omega^0)$ is determined by its value on  $-e_k \in \SS^0$, a minimal geodesics 
$\tilde\gamma(t) = e^{\pi \i A t}J_\ell$  
where $A$ is a self adjoint complex $(\frac{p}{2}\x\frac{p}{2})$-matrix with eigenvalues $\pm 1$ which  commutes  with $J_1, \ldots, J_\ell$.  
Furthermore  {$w(\tilde \gamma) =  \trace_{\C}  A$}. 

As in Remark \ref{kpm} the eigenspaces of $A$ induce an orthogonal splitting 
\[
   \R^p = L_0  \oplus L_1
\]
into subspaces invariant under $J_1, \ldots, J_\ell$ (hence under $\i = J_1 \cdots J_{\ell-1}$)  such that 
\[
    \tilde\gamma(t) = \gamma(t) = e^{\pi \i t} J_\ell\, 
\textrm{ on }L_0 \, , \quad  \tilde\gamma(t) = e^{ -\pi \i t} J_\ell
\, \textrm{ on }L_1 \,  . 
\]
Hence the geodesic $\tilde \gamma $ is of block diagonal form 
\[
      \tilde\gamma = \gamma_0 \oplus \gamma_1 : [0,1] \to \SO(L_0) \times \SO(L_1) 
\]
with midpoint $(+J_k, -J_k)$ and 
$w(\tilde \gamma) = \2(\dim_{\C} L_0-\dim_{\C}L_1)$.
The suspension  $\omega := \Sigma_{J_\ell} (\phi) \in  \Map_*(\SS^1, P_\ell)$ is  
equal to the concatenation
\[
       \omega = \gamma \ast \tilde\gamma^{-1} : [0,1] \to P_\ell \, . 
\]
Therefore  $\omega$ is also in  block diagonal form 
\[
    \omega = \omega_0 \oplus \omega_1 : [0,1]  \to \SO(L_0) \times \SO(L_1) \, , 
\]
where 
\[
     \omega_0 = \gamma_0 \ast \gamma_0^{-1} \,, \quad   
	               \omega_1(t) = e^{2\pi \i t} J_\ell\,. 
\]
The loop $\gamma_0 \ast \gamma_0^{-1}$  is homotopic relative to $\{0,1\}$ to the constant loop with value  $J_\ell$. This motivates the following definition.

\begin{defn} \label{loops} We define the subspace 
\[
      \Map^0_*(\SS^1, P_\ell) \subset \Map_*(\SS^1, P_\ell)
\]
consisting of loops $\omega \in \Map_*(\SS^1, P_\ell)$ such that there is an orthogonal decomposition 
\[
   \R^p = L_0 \oplus L_1
\]
into subspaces which are invariant under $J_1, \ldots, J_\ell$  and with respect to which $\omega$ is in block diagonal form 
\[
    \omega = \omega_0 \oplus \omega_1 : \SS^1 \to \SO(L_0) \times \SO(L_1) 
\]
such that $\omega_0$ is constant with value $J_\ell$ and  
$\omega_1: \SS^1  
\to \SO(L_1)$ is given by 
\[ 
	       e^{2 \pi i t} \mapsto   e^{2\pi \i t}J_\ell \, , 
\]
where $\i = J_1 \cdots J_{\ell-1}$. 
For $0 \leq \eta \leq p/2$ we denote by 
\[
     \Map^0_*(\SS^1, P_\ell)_{\eta} \subset  \Map^0_*(\SS^1, P_\ell)
\]
the subspace of loops with winding number $\eta$
 (equal to $\dim_{\C} L_1 \leq p/2$). 
\end{defn}

We have a canonical homeomorphism 
\beq \label{S1S0}
     h: \Map^0_*(\SS^1, P_\ell) \approx  \Map_*(\SS^0, \Omega^0)  
\eeq
which  replaces the constant map on $L_0$ by the concatenation $\omega = \gamma  \ast \gamma^{-1}$, restricted to $L_0$. 
Since this is homotopic relative to $\{0,1\}$ to the constant loop with value $I \in \SO(L_0)$ by use of the explicit homotopy 
\[
     \omega_s = \gamma_s \ast \gamma_s^{-1} \, , \quad \gamma_s(t) := \gamma(st) \, , \quad 0 \leq s,t \leq 1 \ , 
\]
we conclude that the composition 
\[ 
     \Map^0_\ast(\SS^1, P_\ell) \buildrel h \over \approx  \Map_*(\SS^0, \hat P_k ) \stackrel{\Sigma_\ell}{\longrightarrow} \Map_*(\SS^1, P_\ell)
\]
is homotopic to the canonical inclusion 
\[
    j : \Map^0_*( \SS^1, P_\ell) \to \Map_*(\SS^1, P_\ell) \, . 
\]
Since homotopic maps induce the same maps on homotopy groups and a map $A \to B$ is $d$-connected if and only if it induces a bijection on $\pi_j$ for $0 \leq j \leq d-1$ and a surjection on $\pi_d$, we hence obtain the following version of Proposition \ref{highcon} together with 
Prop.\ \ref{corlk1} where $\Sigma_\ell\o h$ (cf.\ \eqref{S1S0}) is replaced by $j$.

\begin{prop} \label{611} 
Let $d \geq 1$ and $p \geq p_0(d)$ as in Prop.\ \ref{highcon}. 
Assume $s_\ell\.d \leq \eta \leq p/4$.
Then the composition 
\beq \label{theta0}
    \theta^0: \Map^0_\ast(\SS^1,P_\ell)_\eta 
    \buildrel j\over\longrightarrow \Map_*( \SS^1, P_\ell)_{\eta} 
    \stackrel{\theta}{\longrightarrow}    \Map_*(\SS^k , \SO_p)_{\eta} 
\eeq
is $d$-connected.
\end{prop} 

\begin{proof} 
Notice that for $\eta \leq p/4$ the winding number $w = (p/4)-\eta$ is non-negative.  
Hence $p-4|w| = p-4w = 4\eta \geq 4s_\ell\.d$. The assertion now follows from Propositions \ref{highcon} and \ref{corlk1}.
\end{proof} 

By construction, elements $\phi \in  \Map_*(\SS^k , \SO_p)_\eta$ in the image of 
$\theta^0$ in \eqref{theta0} are described as follows. 
There is an orthogonal splitting 
\[
    \R^p = L_0 \oplus L_1
\]
into $\Cl_k$-invariant subspaces with $\dim_{\C} L_1 = \eta$  such that 
\[
      \phi = (\phi_0,\phi_1) : \SS^k \to \SO(L_0)\x \SO(L_1)\subset \SO_p
\]
where $\phi_0 \equiv I$ and $\phi_1 : \SS^k = \SS(\R\oplus\R^k) \to \SS(\End(L_1))$  
is an isometric embedding as a great sphere with image contained in $\SO(L_1)$ which sends
$e_0$ to $I$ and $e_i$ to $J_{i}|_{L_1}$ for $1 \leq i \leq k$. 

Proposition \ref{clifrep} then implies that the map $\phi_1$ is in fact equal to the Hopf map associated to the positive $\Cl_k$-representation $\rho$ restricted to $L_1$. 

\begin{defn} \label{hopfrho}
We call maps $\phi$ of this kind (that is $\phi = (\phi_0,\phi_1)$ on $L_0\oplus L_1$ with $\phi_0 \equiv I$ and $\phi_1$ a positive Hopf map on $L_1$) {\em affine Hopf maps} of 
stable winding number $\eta = \dim_{\C}L_1$ associated to $\rho$ (a positive Clifford representation). 
Let 
\[
      \Hopf_\rho(\SS^k,\SO_{p})_{\eta}  \subset  \Map_*(\SS^k , \SO_p)_{\eta} 
\]
be the subspace of affine Hopf maps of stable winding number $\eta$. 
\end{defn}

Summarizing we obtain the following result on deformations of mapping spaces
(recall $s_k = s_\ell$ from \eqref{Mk}). 

\begin{thm} \label{mindef} 
Let $d \geq 1$ and $k = 4m-1$. 
Then for all $p\geq p_0(d)$ and $s_k\.d \leq \eta \leq p/4$ and any positive Clifford representation $\rho : \Cl_k \to \R^{p\x p}$ the canonical inclusion 
\beq \label{HopfS}
	\Hopf_{\rho}(\SS^k,\SO_{p})_{\eta} \subset\Map_*(\SS^k,\SO_{p})_{\eta} 
\eeq
is $d$-connected. 
\end{thm}

\proof
The theorem follows from Prop.\ \ref{611} since $\Hopf_{\rho}(\SS^k,\SO_{p})_{\eta}$ is the image of $\theta^0$ in \eqref{theta0}.
\endproof

\section{Vector bundles over sphere bundles: Clutching construction} \label{spinor}

The deformation theory for maps $\phi : \SS^k \to \SO_{p}$ can be applied to the classification of oriented Euclidean vector bundles over $n$-spheres where from now on 
$$
	n=k+1 = 4m\,.
$$ 
Choose $N = e_{0}$ and $S = -e_{0}$ (``north and south poles'') in $\SS^n\subset\R^{n+1}$
where $e_0,\dots,e_n$ is the standard basis of $\R^{n+1}$. 
Let 
\bea
	\D^n_\pm &=& \{v\in\SS^n: \pm\<v,e_{0}\> \geq 0\},\cr
	\SS^{n-1}&=& \{v\in \SS^n: \<v,e_{0}\> = 0\} = \D^n_+\cap\D^n_- \nonumber
\eea
(``hemispheres and equator''). 
Let $\mathscr{E}\to\SS^n$ be an oriented Euclidean vector bundle with fibre $\R^{p}$. 
Since $\D^n_\pm$ is contractible, $\mathscr{E}|_{\D^n_\pm}$ is a trivial bundle, isomorphic to $\D^n_\pm\x E_\pm$ where 
\[
   E_\pm = \mathscr{E}_{\pm e_{n}} \cong \R^{p} \, , 
\] 
pulled back to $\D_{\pm}^n$ along the radial projections onto the midpoint. 
The bundle $E$ is obtained from these trivial bundles by identifying $\{v\}\x E_+$ to $\{v\}\x E_-$ via an oriented orthogonal map, thus by a mapping $\phi : \SS^{n-1}\to \SO(E_+,E_-)\cong \SO_{p}$ called {\it clutching function} which is well defined by $\mathscr{E}$ 
up to homotopy. In this situation we write 
\[
     \mathscr{E} = (E_+ , \phi, E_-) \, . 
\]

\begin{ex} \label{clrepsphere} 
Let $\rho: \Cl_n \to \End(L)$ be an orthogonal $\Cl_n$-representation. 
Since $n\!=\!4m$, the longest Clifford product $\omega := e_1\cdots e_n$ (the ``volume element'') commutes with all elements of $\Cl_n^+$ (the subalgebra
spanned by the products of even length) and has order two, $\omega^2 = 1$.\footnote
	{12341234\,=\,--11234234\,=\,234234\,=\,223434\,=\,--3434\,=1, 12345678\,=\,56781234.}
The $(\pm1)$-eigenspaces\footnote
  {Superscripts $\pm$ will always refer to a splitting induced by the volume element.} 
$L^\pm$ of $\rho(\omega)$  are invariant under $\Cl_n^{+}$, and $\rho(v)$  interchanges the eigenspaces for any $v\in\SS^{n-1}$. Using $\mu := \rho|_{\SS^{n-1}}$ we obtain a bundle $\mathscr{L}\to\SS^n$,
\[
\label{hopfdefalt} 
	\mathscr{L} = (L^+,\mu,L^-) 
\]
This is called the {\em Hopf bundle} over $\SS^n$ associated to the Clifford representation $\rho$.
\end{ex}

\begin{rem} \label{Lplus} 
This bundle is isomorphic to $(L^+, \mu_+, L^+)$, where 
\[
     \mu_+ := \mu(e_1)^{-1} \. \mu : \SS^{n-1} \to \SO(L^+) \cong \SO_p
\]
is the Hopf map $\mu_+ = \rho_+|_{\SS^k}$ ($k=n-1$) associated to the representation $\rho_+ : \Cl_k \cong \Cl_n^+ \to \SO(L_+)$,
\begin{equation}  \label{rho+}
\rho_+ : e_i \mapsto \rho(-e_1e_{i+1}), \ \ i = 1, \ldots, k. 
\end{equation} 
Note that $\rho_+$ is positive,\footnote 
	{E.g.\ $k\!=\!7$: Putting $\mu(e_i) =:i$ we have
	$$\mu_+ (e_1)\cdots\mu_+(e_7) = 
	(-1)^7 12\underline{131}4\underline{151}6\underline{171}8 
	= -12\underline{3}4\underline{5}6\underline{7}8 = -\id_{L^+}\ \textrm{ on } L^+.$$} 
that means  its  Clifford family $J_i = \rho(e_i)$, $i = 1, \ldots, k$, on $L^+$ satisfies 
\[
    J_k = J_1 \cdots J_{k-1} \, . 
\]
In particular the stable winding number of $\mu_+$ (see Definition \ref{windingloop})  is given by 
\[
   \label{formulastable}    \eta^s (\mu_+) = \2  \dim L^+  = \frac{1}{4}\dim L  > 0 \, . 
\]
Compare Example \ref{winding_for_hopf}.
 The positivity of this winding number will become important later in the proof of Theorem \ref{main}. 
\end{rem} 

We now replace the sphere $\SS^n$ by a sphere bundle with two antipodal sections over a finite CW-complex $X$.
More specifically, starting from an $n$-dimensional Euclidean vector bundle $V\to X$ we glue two copies $\D_{\pm}  V$ of its disk bundle $\D V \to X$ along the common boundary, the unit sphere bundle $\SS V$, by the identity map. 
Thus we obtain an $\SS^n$-bundle $\hat V = \SS(\R \oplus V)\to X$,
\beq \label{hatV}
	\hat V = \D_+V \cup_{\SS V} \D_-V.
\eeq
We obtain two sections $s_\pm : X\to \hat V$ of the bundle $\hat V \to X$, which we regard as north and south poles,  defined as the zero sections of $\D^\pm V$. 

Let 
\[
    \mathscr{E} \to \hat V
\]
be a Euclidean vector bundle over the total space of $\hat V \to X$.
Let 
\[
     E_\pm = s_\pm^*\mathscr{E} \to X 
\]
which we sometimes tacitly  pull back to bundles $E_\pm\to \D_{\pm}V$ along the canonical fiberwise projection maps $\D_{\pm} V \to s_\pm(X)$. 

Since there are -- up to homotopy unique -- bundle isometries $\mathscr{E}|_{\D_{\pm}  V} \cong E_\pm$ restricting to the identity over $\{s_{\pm}\}$, the bundle $\mathscr{E}$ is obtained from $E_\pm$ by a clutching map $\sigma$, that is a section of the bundle 
\[
     \Map(\SS V,\Or(E_+,E_-)) \to X \, , 
\]
which is uniquely determined up to fiberwise homotopy. 

Hence,  for any $x\in X$ we have a map $\sigma_x : SV_x \to \Or((E_+)_x,(E_-)_x)$. 
Note that we may equivalently consider the clutching map as a bundle isometry $\sigma : E_+|_{\SS V}\to E_-|_{\SS V}$. 
In this situation  we write 
\[
    \mathscr{E} = (E_+,\sigma,E_-) \, . 
\]

Vice versa, if $E_{\pm} \to X$ are oriented Euclidean bundles and 
\[
   \sigma \in \Gamma (   \Map(\SS V,\Or(E_+,E_-)) ) 
\]
a clutching map, then we obtain a vector bundle $\mathscr{E} = (E_+, \sigma, E_-) \to \hat V$ by gluing the pull back bundles  of $E_{\pm} \to \D_{\pm}  V$ along $\SS  V $ by $\sigma$.

\begin{ex} 
A particular case is $E_+ = E_- =: F$ and $\sigma(v) = \id_{F_v}$ for all $v\in \SS V$. 
We write $(F,\id,F)$ for this triple.
Since $F_v$ can be identified with $F_{s_{\pm}}$ for all $v \in \D^\pm V$, this bundle is trivial over every fibre  $\hat V_x$, $x \in X$, hence it is isomorphic to a bundle over $X$, pulled back to $\hat V$ by the projection $\pi : \hat V \to X$. 

Vice versa, for any vector bundle $F \to X$, the pull-back bundle $\pi^*F$ is given by the triple $(F,\id,F)$. 
\end{ex}

For a generalization of Example \ref{clrepsphere} to sphere bundles we recall the following definition.

\begin{defn} \label{Lambda}
Let $\Cl(V) \to X$ be the Clifford algebra bundle associated to the Euclidean bundle $V$ with fibre $\Cl(V)_x = \Cl(V_x)$ over $x \in X$. 

A real  {\em $\Cl(V)$-Clifford module bundle} is a vector bundle $\Lambda \to X$ such that each fibre $\Lambda_x$, $x \in X$, is a real $\Cl(V_x)$ module. 
More precisely, we are given a bundle homomorphism $\mu : \Cl(V) \to \End(\Lambda)$ which restricts to an algebra homomorphism in each fibre. 

We may and will assume that $\Lambda$ is equipped with a Euclidean structure such that  Clifford multiplication with elements in $\SS V$ is orthogonal. 
\end{defn}

{From now on we will in addition assume that $V \to X$ is oriented.} 
Any oriented orthonormal frame\footnote 
	{A frame (basis) $b = (b_1,\dots,b_n)$ of $V_x$ will be considered
	as a linear isomorphism $b : \R^n \to V_x$, $v \mapsto bv = \sum_j v_jb_j$.}
of $V|_U$ over $U\subset X$ induces an orientation preserving orthogonal trivialization  
\[
    b: U\x \R^n \cong  V|_U \, . 
\]
Thus $L := \Lambda_{x_o}$ with some fixed $x_o\in U$ becomes a $\Cl_n$-module.
Consider the bundle 
\[ 
     \Iso_{\Cl_n}(\underline{L},\Lambda|_U) \to U 
\]
whose fibre over $x\in U$ is the space of $\Cl_n$-linear isomorphisms $\phi_x : L\to\Lambda_x$. 
If $U$ is contractible it has a section $\phi$ which intertwines the $\Cl_n$-module multiplications (denoted by $\.$) on $L$ and $\Lambda_x$, $x\in U$: for all $\xi \in L$ and $v\in \R^n$ and $x\in U$ we have
\beq \label{intertw}
	  \phi_x(v\.\xi) = b_x(v)\.\phi_x(\xi).
\eeq
In other words: We obtain Euclidean trivializations 
\begin{equation}  \label{triv}
     V|_U \cong U \x\R^n, \qquad
   \Cl(V)|_U  \cong U \times \Cl_n  , \quad 
   \Lambda|_U \cong  U \times L
\end{equation} 
with a $\Cl_n$-module $L$ such that the $\Cl(V)$-module structure on $\Lambda$ corresponds to the $\Cl_n$-multiplication on $L$. 
Notice that in particular the $\Cl_n$-isomorphism type of $L$, the {\em typical fibre}  of $\Lambda \to X$  is uniquely determined over each path component of $X$. 

Let $n = 4m$. 
Since $V \to X$ is oriented, the volume element $\omega \in \Cl_n$ defines a section (``volume section'') of $\Cl(V) \to X$, and the $\pm 1$-eigenspaces define the positive and negative Clifford algebra bundles $\Cl^{\pm}(V)  \to X$ and subbundles $\Lambda^{\pm} \to X$ of $\Lambda \to X$, which are invariant under $\Cl^+(V)$ and such that for every $v \in V_x$ the endomorphism $\mu(v)$ interchanges $\Lambda_x^+$ and $\Lambda_x^-$. 
Correspondingly the typical fibre $L$ decomposes as $L = L^+ \oplus L^-$ such that the local trivializations 
\eqref{triv} preserve positive and negative summands.

\begin{defn} \label{hopfbun} 
For $n = 4m$ the bundle $\mathscr{L}\to\hat V$ defined by the triple
\beq
	\mathscr{L}  = (\Lambda^+,\mu,\Lambda^-)
\eeq 
is called the {\em Hopf bundle} associated to the $\Cl(V)$-Clifford module bundle $\Lambda \to X$. 
Note that each $\mathscr{L}|_{\hat V_x} \to \hat V_x $, $x \in X$,  is a  Hopf bundle  in the sense of Example \ref{clrepsphere}  after passing to local trivializations. 
\end{defn}

\begin{ex} \label{ClHopf} As a particular example we consider the $\Cl(V)$-module bundle 
\[
   \Lambda = \Cl(V) \to X
\]
 with $\Cl(V)$-module structure $\mu$ given by (left) Clifford multiplication on $\Cl(V)$. 
Note that $\Lambda^{\pm} = \Cl^{\pm}(V) \to X$ 
where $\Cl^+(V)$ ($\Cl^-(V)$) is generated as a vector bundle by the Clifford products of even (odd) length.

We define the {\em Clifford-Hopf bundle} $\mathscr{C} \to \hat V$ as the triple 
\beq \label{cliffhopf}
	\mathscr{C} = (\Cl^+(V),\mu,\Cl^-(V))
\eeq
where $\mu : \SS V \to \SO(\Cl^-(V),\Cl^+(V))$ is the (left) Clifford multiplication.
\end{ex}

\begin{ex} \label{spinorb}
Recall  that a {\em spin structure} on $V \to X$ is given by  a two fold cover
\[
    P_{\Spin}(V) \to P_{\SO}(V)
\]
of the $\SO(n)$-principal bundle of oriented orthonormal frames in $V$, which is equivariant with respect to the double cover $\Spin_n \to \SO_n$. 

Consulting Theorem \ref{perclif} we see that for $n = 4m$ there is exactly one irreducible $\Cl_n$-module $S = S_n$, and we obtain the {\em spinor bundle} 
$$
	\Sigma  = P_{\Spin}(V) \x_{\Spin_n} S.
$$
Since the Clifford algebra bundle can be written in the form 
\[
    \Cl(V) = P_{\Spin}(V) \x_{\Spin_n} \Cl_n
\]
where $\Spin_n\subset \Cl_n$ acts on $\Cl_n$ by conjugation, $\Sigma$ becomes a $\Cl(V)$-module bundle by setting
\[
     [p,\alpha]\.[p,\sigma] = [p,\alpha\sigma] 
\]
for $p\in P_{\Spin}(V)$, $\alpha\in \Cl_n$ and $\sigma\in S$.
\end{ex} 

\begin{defn} \label{spinorhopf} 
The {\em Spinor Hopf bundle over} $\hat V$ is the Euclidean vector bundle $\mathscr{S} \to \hat V$ defined by the triple
\beq \label{hopfdef2}
	\mathscr{S} = (\Sigma^+,\mu,\Sigma^-)
\eeq 
where $\mu : \SS V \to \SO(\Sigma^+,\Sigma^-)$ is the restriction of the Clifford  module multiplication. 

\end{defn}

\begin{defn} 
Let $\Lambda$ be a $\Cl(V)$-module bundle with Clifford multiplication $\mu$ and a $\Cl(V)$-invariant orthogonal decomposition
$$
	\Lambda = \Lambda_0 \oplus \Lambda_1
$$
into $\Cl(V)$-submodule bundles.
This induces a triple 
\begin{equation} \label{afftrip} 
    \big(\Lambda^+ , \omega,   \Lambda_0^+  \oplus \Lambda_1^-\big) ,\ \ 
	\omega = \id_{\Lambda_0^+}\oplus\mu|_{\Lambda_1^+} 
\end{equation} 
equal to the direct sum bundle
\[
 (\Lambda_0^+,\id,\Lambda_0^+)  \oplus  (\Lambda_1^+, \mu , \Lambda_1^-) \to X \, . 
\]
We call $(\Lambda^+ , \omega,  \Lambda_0^+  \oplus \Lambda_1^-)$ the  {\em affine Hopf bundle} associated to the decomposition $\Lambda = \Lambda_0 \oplus \Lambda_1 \to X$.
\end{defn}

In the next section we will prove that {\em every} vector bundle over $\hat V$ is -- after suitable stabilization -- isomorphic to an affine Hopf bundle. 

\section{Vector bundles over sphere bundles and affine Hopf bundles} 
 \label{difficult} 

 Let $V \to X$ be an oriented Euclidean vector bundle of rank $n = 4m$ with associated sphere bundle $\hat V = \SS( \R \oplus V) \to X$. 

\begin{defn} \label{stabiso} 
\begin{enumerate} 
   \item Two {Euclidean} 
vector bundles $E,\tilde E$ over $X$ or $\hat V$  are called {\em stably isomorphic}, written 
   \[
            E\cong_s\tilde E \, , 
   \]
       if $E\oplus\underline{\R}^q \cong \tilde E\oplus \underline{\R}^q$ for some trivial vector bundle $\underline{\R}^q\to X$.
   \item Two  $\Cl(V)$-module bundles $\Lambda,\tilde \Lambda \to X$ are  called {\em stably isomorphic}, written $\Lambda \cong_s\tilde\Lambda$, if there is some $\Cl(V)$-linear isomorphism 
\[
       \Lambda \oplus \Cl(V)^q \cong \tilde\Lambda \oplus \Cl(V)^q 
\]
for some $q$.        
\end{enumerate} 
\end{defn}

The following result is central in our paper. 

\begin{thm} \label{main}
Let 
\[
    \mathscr{E} \to \hat V 
\]
be a Euclidean vector bundle. 
Then   possibly after adding copies of $\mathscr{C}$ (the Clifford-Hopf bundle, see Example \ref{ClHopf}) and  trivial vector bundles $\underline{\R}$ to $\mathscr{E}$ we have 
\beq \label{0+1H}
	\mathscr{E}  \cong {E}\oplus \mathscr{L} 
\eeq 
for a Euclidean vector bundle ${E}\to X$ (pulled back to $\hat V$) and a  Hopf bundle $\mathscr{L}\to \hat V$ associated to some $\Cl(V)$-module bundle $\Lambda \to X$.

Let  ${E,\tilde E}\to X$ be Euclidean vector bundles  and $\mathscr{L},\mathscr{\tilde L} \to \hat V$ Hopf bundles  corresponding to $\Cl(V)$-module bundles  $\Lambda,\tilde\Lambda\to X$. 
If  
\begin{equation} \label{unique} 
     {E} \oplus \mathscr{L} \cong_s {\tilde E} \oplus \mathscr{\tilde L} 
\end{equation}
then ${E} \cong_s {\tilde E}$ and ${\Lambda} \cong_s {\tilde\Lambda}$. 
\end{thm}

\begin{rem} \label{83}
When $V$ carries a spin structure, we will see in Prop.\ \ref{meven} (which is self consistent) 
that $\Lambda \cong E\otimes\Sigma$ when $m$ is even, and
$\Lambda \cong E\otimes_{\H}\Sigma$ when $m$ is odd, where 
{$E$ is some Euclidean vector bundle over $X$ and} 
$\Sigma$ denotes the Spinor bundle
associated to $P_{{\Spin}}(V)$, cf. Example \ref{spinorb}. Therefore 
$$
	\mathscr L \cong \left\{\begin{matrix} 
		E\otimes\mathscr S & \textrm{ when $m$ is even} \cr
		E\otimes_{\H}\mathscr S & \textrm{ when $m$ is odd}\end{matrix}\right.
$$
where $\mathscr S$ denotes the spinor-Hopf bundle (cf. Definition \ref{spinorhopf})
and $E$ a vector bundle over $X$, pulled back to $\hat V$. Further, the stable isomorphism type
of $\Lambda$ determines the stable isomorphism type of $E$, see Lemma \ref{95}.
\end{rem}

\bigskip
The proof of Theorem \ref{main} will cover the remainder of this section.
We represent $\mathscr{E}$ by a triple $(E_+, \sigma, E_-)$ for Euclidean vector bundles $E_\pm\to X$, which we consider as bundles over $\D_{\pm}  V$ in the usual way.

\begin{rem} \label{homopm}
Let $\mathscr{E}  = (E_+,\sigma,E_-)$ and $\tilde{\mathscr E}=(\tilde E_+,\tilde\sigma,\tilde E_-)$ be triples. 
Then an orthogonal vector bundle homomorphism  
\[
   \phi : \mathscr{E} \to \tilde{\mathscr E} 
\]
consists of a pair $(\phi_+, \phi_-)$ of orthogonal bundle homomorphisms $\phi_\pm : E_\pm \to \tilde E_\pm$ over $\D V$ such that the following diagram over $\SS V$ commutes:
\beq \label{isomorph}
\begin{matrix}
   \xymatrix{  E_+|_{\SS V}   \ar@{>}[d]_{\sigma}   \ar@{>}[r]^{\phi_+}  & 
	\tilde E_+|_{\SS V}   \ar@{>}[d]^{\tilde\sigma}   \\
  	E_-|_{\SS V}    \ar@{>}[r]^{\phi_-}  &   \tilde E_-|_{\SS V}  } \end{matrix}
\eeq
If $\phi_+^t : E_+ \to \tilde E_+$, $t\geq 0$, is a homotopy of bundle isomorphisms over $\D V$ with $\phi_+ = \phi_+^0$, we may replace $\phi$ by $\phi^t = (\phi_+^t,\phi_-)$, 
replacing the clutching map 
$\tilde\sigma = \phi_-\sigma\phi_+^{-1}$ over $\SS V$ by the homotopic clutching map 
$\tilde\sigma^t = \tilde\sigma\phi_+(\phi_+^t)^{-1}$ which does not change $\tilde{\mathscr E}$.
\end{rem} 

\bigskip
In the proof of Theorem \ref{main} we can assume without loss of generality that $X$ is path connected. 
Let 
\[
    V|_U \cong U \times \R^n \, , \quad {E_+}|_U \cong U \times \R^p 
\]
be orthogonal trivializations over a connected open subset $U \subset X$, where the first trivialization is assumed to be orientation preserving. 
(We are not assuming that the the bundle ${E_+} \to X$ is orientable). 
We obtain an induced trivialization 
\[
      E_-|_U \cong U \times \R^p 
\]
using the isomorphisms $\sigma_x(e_1) : (E_+)_x \to (E_-)_x$ for $x \in U$ where $e_1 \in \R^n$ is the first standard basis vector. 
With respect to theses trivializations $\sigma|_U$ is given by a map 
\beq \label{sigmaU}
     \sigma|_U : U \to \Map_*( \SS^{n-1} , \SO_p) \, . 
\eeq
Here we recall that $\SS^{n-1}$ is connected since $n \geq 4$ by assumption so that the local clutching function has indeed values in $\SO_p$ rather than $\Or_p$. 
The stable winding number $\eta^s(\sigma_x) \in \Z$ (see Definition \ref{windingloop}) is independent of the chosen trivializations and constant over $U$. 
It is hence an invariant of the triple $\mathscr{E} = (E_+, \sigma, E_-)$. 

Set $d = \dim X$. 
After adding trivial bundles and Clifford-Hopf bundles 
to $\mathscr{E}$ we can assume that the stable mapping degree $\eta = \eta^s( \sigma)$ satisfies the conditions 
\beq \label{d,p} 
	s_k\.d \leq \eta \leq p/4
\eeq 
from Theorem \ref{mindef} (with $k := n-1$, recall $s_k = s_\ell$ for $\ell = k-1$), compare Remarks \ref{windingadditive} and \ref{Lplus}. 
Furthermore, we can assume that $p$ is divisible by $s_k$. 
In particular, by Theorem \ref{mindef}, with these choices of $p$ and $\eta$ the canonical inclusion 
\[
	\Hopf_{\rho}(\SS^k,\SO_{p})_{\eta} \subset\Map_*(\SS^k,\SO_{p})_{\eta} 
\]
is $d$-connected for any positive representation $\rho : \Cl_k \to \End(\R^p)$ (which is equivalent to the $(p / s_k)$-fold direct sum of the positive $S_k$).
We will work under these assumptions from now on. 

Next we note that we may add on both sides of \eqref{0+1H}  a bundle $F\to X$, pulled back via $\pi:\hat V\to X$ (which further increases $p$, but not the winding number $\eta$, so that the assumptions of Theorem \ref{mindef} are preserved), that is we add $\pi^*F = (F,\id,F)$. 
Then the original  statement is obtained by embedding $F$ into a trivial bundle over $X$ and adding a complement  $F^\perp$ on both sides of \eqref{0+1H}.

We use this freedom to put $E_+$ into a special form.
Since $E_+$ embeds into a trivial bundle $\underline{\R}^N \to X$ and each bundle over $X$ of rank larger than $\dim (X)$ splits off a trivial line bundle (its Euler class vanishes), 
$E_+$ embeds into any sufficiently large vector bundle over $X$, 
in particular into  ${\Lambda^+}$ for some   $\Cl(V)$-module bundle $\Lambda\to X$ (e.g.\ $\Lambda = \R^q\otimes \Cl(V) = \Cl(V)^q$).
Then $E_+$ is a direct summand of $\Lambda^+$. 

After adding $(F, \id, F)$ to $(E_+, \sigma, E_-)$ with  $F = (E_+)^{\perp} \subset \Lambda^+$ we may hence assume that  
$$
	E_+ = \Lambda^+\,.
$$ 
Note that this bundle is oriented. 
After these preparations Theorem \ref{main}  is proven by induction over a cell decomposition of $X$. 
The decomposition \eqref{0+1H} results from the following fact.

\begin{prop} \label{important}  
There is an orthogonal decomposition of $\Lambda^{\pm}$  into $\Cl(V)$-invariant subbundles
\beq \label{E0+E1}
	\Lambda^\pm  = \Lambda^\pm_0 \oplus \Lambda^\pm_1
\eeq
with the following property: 
the triple $\mathscr{E} = (E_+, \sigma , E_-) $  is isomorphic to the  triple 
\beq \label{hopftriple}
    \big(\Lambda_0^+ \oplus \Lambda_1^+ ,\id\oplus\mu, \Lambda_0^+ \oplus \Lambda_1^- \big)
	= (\Lambda_0^+ , \id , \Lambda_0^+) \oplus (\Lambda_1^+,\mu,\Lambda_1^-) 
\eeq
where $\mu : \SS V \x \Lambda_1^+ \to \Lambda_1^-$ is the $\Cl(V)$-module multiplication on  $\Lambda_1$.

Furthermore this isomorphism can be chosen as the identity on the first bundle  $E_+ = \Lambda^+$. 
Hence, by \eqref{isomorph}, it is given by an isomorphism  of vector bundles over the total space $\D V$, 
$$
	f = f_-: E_- \buildrel \cong\over\longrightarrow \Lambda_0^+  \oplus \Lambda_1^- 
$$ 
such that over $\SS V$ we have
\[
    f\o\sigma = \omega := \id\oplus\mu.
\]
\end{prop}

\begin{proof} 
In the induction step let $X = X' \cup D$ be obtained by attaching a cell $D$ to $X'$  and assume that the assertion holds for the restriction to $X'$ of the bundle 
$$
	\mathscr{E} = (E_+,\sigma,E_-)
$$
We denote by $V' \to X'$ the restriction of $V$. 
Hence we have the decomposition \eqref{E0+E1} over $X'$ and an isomorphism 
\beq \label{isof}
     f : E_-|_{\D V'} \buildrel \cong\over\to 
	(\Lambda_0^+  \oplus  \Lambda_1^-)|_{\D V'} 
\eeq
such that over $\SS V'$,
\beq \label{fsigma}  
    f \circ \sigma = \omega = \id\oplus\mu \in 
	\SO(\Lambda_0^+\oplus\Lambda_1^+,\Lambda_0^+\oplus\Lambda_1^-)  
\eeq
where $\mu(v)\in\SO(\Lambda_1^+,\Lambda_1^-)$ is the $\Cl(V)$-module multiplication ($v\in\SS V'$).

\medskip
We need to extend $f$ to a similar isomorphism over $X = X' \cup D$.
In particular we need to extend the bundle decomposition  $\Lambda = \Lambda_0 \oplus \Lambda_1$ from  $X'$ to $X$. 
This will ultimately be achieved by applying Theorem \ref{mindef} to the restriction of $\mathscr{E} \to\hat V$  to $\hat V|_{D}\cong D\x\SS^n$. 
In fact, trivializing ${E_\pm}$ over $D \subset X$, the clutching map $\sigma$ of $\mathscr{E}$ will become a map $\hat\sigma : D \x \SS^{n-1} \to \SO(L^+)$ where $L^+ = \R^p$ is the standard fibre of ${E_+}$. 
However, in order to apply Theorem \ref{mindef}, this map $\hat\sigma : D\to \Map(\SS^{n-1},\SO(L^+))$ needs
\begin{itemize} 
\item[(i)] taking values in $\Map_*(\SS^{n-1},\SO(L^+)) = \{\phi: \SS^{n-1}\to\SO(L^+):\phi(e_1) = I\}$ (see Def.\ \ref{mapstar}), 
\item[(ii)] $\hat\sigma(\d D) \subset \Hopf_{\rho_+}(\SS^{n-1},\SO(L^+))_\eta$ with $\rho_+ : \Cl_k\to \End(L_+)$ as in \eqref{rho+}. 
\end{itemize}
Requirement (ii) will be met (using the clutching maps $\sigma,\omega$ as trivializations) by transforming $f$ to $\hat f$, see \eqref{hatf}, which will change \eqref{fsigma} to \eqref{hatfsigma}. 
From \eqref{hatfsigma} we will see using Remark \ref{Lplus} that the clutching map $\hat f\hat\sigma:\d D \to\Map_*(\SS^{n-1},\SO(L^+))$ takes values in $\Hopf_{\rho_+}(\SS^{n-1},\SO(L^+))$.
In order to meet the requirement (i) we need an extension $\hat F$ of $\hat f$ from $\d D$ to all of $D$ such that $\hat F(x,e_1) = \id_{L_+}$ for all $x\in D$, see the ``Assertion'' below. 
Then we can apply our deformation theorem \ref{mindef} to the clutching map $\tau = \hat F\hat\sigma$, thus obtaining a new clutching map $\tau_1$ with values in $\Hopf_\rho(\SS^{n-1},\SO(L^+))_\eta$. It fits together with the given clutching map along $X'$ since there was no change along $\d D = X'\cap D$.

\medskip
Now we explain these steps in detail.
Choose a trivialization of $V|_D$, that is an oriented orthonormal frame over $D$.  
It induces trivializations of $\D V$ and $\Cl(V)$. 
Also we choose a compatible trivialization of $\Lambda$, compare Example \ref{Lambda}:
\beq \label{trivial}
     \D V|_D \cong D\x\D^n, \qquad
   \Cl(V)|_D  \cong (\underline{\Cl}_n)_D   , \quad 
   \Lambda|_D \cong (\underline{L})_D
\eeq
for some $\Cl_n$-module $L=L^+\oplus L^-$ (recall $n = 4m$) such that for any $x\in D$ the  $\Cl(V_x)$-module multiplication $V_x \x \Lambda_x \to \Lambda_x$ is transferred to the $\Cl_n$-module multiplication $\R^n\x L \to L$.

Identifying $\SS V|_D$ with $D\x\SS^{n-1}$, we choose $e_1\in\SS^n$ as a base point  and put 
\beq 
     \sigma_1(x) := \sigma(x,e_1) , \quad \omega_1(x) := \omega(x,e_1) \, . 
 \eeq
We use these bundle isomorphisms as trivializations $\eta$ and $\vartheta$ of the bundles $E_-|_{D}$ and $(\Lambda_0^+\oplus \Lambda_1^-)|_{\d D}$
(which below will be pulled back to $\D V|_D$ and $\D V|_{\d D}$, respectively): 
\beq \label{eta}
\begin{matrix}
  \eta &:& E_-
	&\buildrel\sigma_1^{-1}\over\longrightarrow& 
	E_+ &=& \Lambda^+ \cong \underline{L}^+ &\textrm{over } D,&&  \cr
   \vartheta &:&  \Lambda_0^+ \oplus \Lambda_1^- 
	&\buildrel \omega_1^{-1}\over\longrightarrow&  
	\Lambda_0^+ \oplus \Lambda_1^+ &=& \Lambda^+ \cong \underline{L}^+&\textrm{over } \d D.&& 
\end{matrix}
\eeq
Thus we can express the isomorphism $f$ in \eqref{isof}, restricted to $\D V|_{\d D}$, as a map 
\beq \label{hatf}
\hat f = \vartheta\o f\o\eta^{-1}: \d D \x \D^{n} \to \SO(L^+).
\eeq
At the base point $e_1\in\SS^{n-1}$ we obtain for all $x\in\d D$:
\bea \label{hatfe1}
	\hat f(x,e_1) &=& \vartheta(f(\sigma_1(x))) \\
	&\buildrel \eqref{fsigma}\over =& \vartheta(\omega_1(x)) \cr
	&\buildrel \eqref{eta}\over =& 
	\omega_1(x)^{-1}\o \omega_1(x) \cr
	&=& \id_{L^+}. \nonumber
\eea
Further, \eqref{fsigma} is transformed over $\SS V|_{\d D} = \d D\x\SS^{n-1}$ into
\beq \label{hatfsigma}
	\hat f  \hat\sigma = \hat\omega : \d D \x \SS^{n-1}\to \SO(L^+)
\eeq
where $\hat\sigma = \eta\o\sigma : D\x\SS^{n-1}\to \SO(L^+)$ and $\hat\omega = \vartheta\o\omega : \d D\x \SS^{n-1}\to SO(L^+)$ are evaluated as follows:
\beq \label{hatsigma}
\begin{matrix}
	\hat\sigma(x,v) &=& \sigma(x,e_1)^{-1}\sigma(x,v) &\textrm{for all}& x\in D,&&\cr
	\hat\omega(x,v) &=& \omega(x,e_1)^{-1}\omega(x,v) &\textrm{for all}& x\in \d D.&&
\end{matrix}
\eeq
	
Now we will extend $\hat f$ from $\d D\x\D^{n}$ to $D\x\D^{n}$.
The extended map 
\[
    \hat F: D\x\D^{n}\to \SO(L^+)
\]
restricted  to $D\x\SS^{n-1}$, will be used to define another clutching map 
\beq \label{tau}
	\tau := \hat F\hat\sigma : D\x\SS^{n-1}\to \SO(L^+) \
	\textrm{ with $\tau =\hat f\hat\sigma=\hat\omega$ on $\d D\x\SS^{n-1}$.}
\eeq

\bigskip 

\begin{ass} 
The map $\hat f : \d D\x\D^n \to \SO(L^+)$ in \rm{\eqref{hatf}} extends to a map 
\[
     \hat F :D\x\D^n \to \SO(L^+)
\]
 such that $\hat F(x,e_1) = \id_{L^+}$ for all $x\in D$. 
\end{ass}

\begin{proof}[Proof of assertion] 
By \eqref{hatfe1}, $\hat f$ is equal to the identity on the fibre over  $\partial D \times \{e_1\}$, where as usual $e_1 \in \SS^{n-1} \subset \D^{n}$. 
Notice that $\{e_1\} \subset \D^{n}$  is a strong deformation retract. 
Let  $r_t :  \D^{n} \to \D^{n}$, $t\in[0,1]$,  be a deformation retraction with $r_1 = \id_{\D^n}$ and $r_0 \equiv e_1$: 

\ms\hskip 5cm
\includegraphics{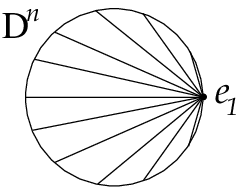}

\ms
This gives a homotopy  $\hat f_t: \id_{L^+}\simeq \hat f$,
$$
    \hat f_t = \hat f \o (\id_{\d D},r_t) : \partial D \times \D^{n}  \to  \SO(L^+)	
$$
with $\hat f_t(\tilde x,e_1) = \id_{L^+}$ for all $\tilde x\in \d D$ and $t\in[0,1]$.
Considering $D$ as the cone over $\d D$, this can be viewed  as a map
\bea
   \hat F :   D \times \D^{n} &\to& \SO(L^+) \,,\cr
		\hat F(x,v) &=& \hat f_{|x|}(\tilde x,v) = \hat f(\tilde x,r_{|x|}(v)) 
\nonumber
\eea
with $\tilde x = x/|x|$ when $x\neq 0$. Moreover, $F(0,v) = \id_{L^+}$ for all $v\in\D^n$
since $r_0 \equiv e_1$ and $\hat f(.,e_1)\equiv \id_{L^+}$.
This map restricts to $\hat f$ over $\partial D \times \D^{n}$ and it is 
constant $=\id_{L^+}$ along the fibers over 
$D \times [0,1]e_1 \subset D \times \D^{n}$, thus finishing the proof of the assertion. 
\end{proof}

The new clutching map $\tau =\hat F\hat\sigma : D \to \Map(\SS^{n-1}, \SO(L^+))$ satisfies
$$
	\tau(\d D) \subset \Hopf_{\rho_+}(\SS^{n-1},\SO(L^+))_\eta
$$
by \eqref{tau}, \eqref{hatsigma}, \eqref{fsigma} (cf.\ Definition \ref{hopfrho} and \eqref{rho+}),
where $\eta$ is the stable winding number of the original clutching map $\sigma$ in the sense of \eqref{sigmaU}.

Since 
	$p \geq p_0(d)\ \textrm { where }\ d = \dim X \textrm{ and } s_\ell\.d \leq \eta \leq p/4$, 
we can apply Theorem   \ref{mindef} to obtain a deformation $(\tau_t)_{t\in[0,1]}$ of $\tau=\tau_0$  such that $\tau_t = \tau$ on $\d D$ and
\[
     \tau_1(D)\in\Hopf_{\rho_+}(\SS^{n-1},\SO(L^+))_\eta.
 \]
Thus the trivial bundle $\underline{L}^+$ decomposes over $D$ into orthogonal $\Cl_{n-1}$-invariant subbundles, 
\beq \label{L+}
	\underline{L}^+|_D = L_0^+\oplus L_1^+ \ 
	\textrm{ and }\ \tau_1   = \id \oplus \hat\mu\,,	
\eeq 
where $\hat\mu \buildrel \eqref{hatsigma}\over = \mu(e_1)^{-1}\mu = \rho_+|_{\SS^{n-1}}$ is the Hopf map induced by the $\Cl_{n-1}$-module structure on $L^+_1$ (compare Remark \ref{Lplus}). 

We have to connect these clutching data over $D$ to the given ones over $X'$. 
In \eqref{trivial}
we have chosen a trivialization $\Lambda^+|_D \buildrel \cong\over\to (\underline{L}^+)_D$  which over $\d D$ transforms the decomposition 
$\Lambda^+ = \Lambda_0^+\oplus \Lambda_1^+$ into a bundle decomposition 
$(\underline{L}^+)_{\d D} = (L_0^+\oplus L_1^+)|_{\d D}$\,. 
Now we extend the decomposition \eqref{E0+E1} 
of $\Lambda^+|_{\d D}$ to all of $D$ by 
applying the inverse of this trivialization to \eqref{L+} obtaining a decomposition $\Lambda^+ = \Lambda^+_0 \oplus \Lambda_1^+$ on all of $X$. 
Similarly, we obtain subbundles $\Lambda^-_{0} = \mu(V)\Lambda^+_{0}$ and $\Lambda_1^- = \mu(V)\Lambda_1^+$ decomposing $\Lambda^-$.

Now we can define a bundle isomorphism $$F : E_- \to \Lambda_0^+\oplus\Lambda_1^-$$ over $\SS V$, with $F = f$ over $\SS V|_{X'}$ and 
\beq \label{F} 
F : E_- \buildrel \sigma^{-1} \over\longrightarrow E_+ = \Lambda^+ 
\buildrel \id\oplus\mu\over\longrightarrow \Lambda_0^+\oplus\Lambda_1^-\ \textrm{ over } \SS V|_{D}
\eeq
This is compatible over $\d D = X'\cap D$ because  over $\SS V|_{\d D}$ we have $f =\omega \o \sigma^{-1}$ and $\id\oplus\mu = \omega$. 
It defines an isomorphism 
$$
\phi = (\id_{E_+},F) : (E_+,\sigma,E_-)\to (\Lambda^+,\omega,\Lambda_0^+\oplus\Lambda_1^-)\,.
$$ 
In fact, over $X'$ this is true by assumption,  and over $D$ we use \eqref{F} for the commutativity of the diagram 
$$
\begin{matrix}
   \xymatrix{  E_+   \ar@{>}[d]_{\sigma}   \ar@{>}[r]^{=}  & 
	\Lambda^+   \ar@{>}[d]^{\id\oplus\mu}   \\
  	E_-    \ar@{>}[r]^{F}  &   \Lambda_0^+\oplus\Lambda_1^-  } \end{matrix}
$$
Hence the isomorphism $f$ from Equation \eqref{isof} extends to an isomorphism $\phi$ over $X$  between the given bundle $\mathcal E$ and a bundle of the form \eqref{hopftriple}.
This finishes the proof of Proposition \ref{important} and the existence part of  Theorem \ref{main}.
\end{proof} 

For the uniqueness statement of Theorem \ref{main} we need the following lemma. 

\begin{lem} \label{isomorphism} 
Let $\mathscr{E} = (E_+, \sigma, E_-)$ and $\tilde{\mathscr{E}} = ( \tilde E_+, \tilde \sigma, \tilde E_-)$ be triples with 
$$E_+ = \tilde E_+ =:F$$ 
and 
$$
   \phi = (\phi_+, \phi_-) : \mathscr{E} \to  \tilde{\mathscr{E}} 
$$
an orthogonal bundle isomorphism 
(see Remark \ref{homopm}). Let $\mathscr F = (F,\id,F) = \pi^*(F)$.
Then there exists an orthogonal isomorphism 
\[
   \tilde\phi :  \mathscr{E} \oplus \mathscr F \to   \tilde{\mathscr{E}}\oplus \mathscr F
\]
such that $\tilde\phi_+ = \id$.
\end{lem}

\begin{proof} 
We consider the isomorphism $\hat\phi' = (A,B) : \mathscr E \oplus \mathscr F
\to \mathscr F \oplus \tilde{\mathscr E}$ given by the following commutative diagram:
$$
\begin{matrix}
   \xymatrix{  E_+ \oplus F  \ar@{>}[d]_{\sigma\oplus\id}   \ar@{>}[r]^{A}  & 
	F\oplus \tilde E_+  \ar@{>}[d]^{\id\oplus\tilde\sigma}   \\
  	E_-\oplus F    \ar@{>}[r]^{B}  &   F\oplus\tilde E_-} \end{matrix}
$$
with $A = \bpm & -\phi_+^{-1} \cr \phi_+ \epm$ and $B = \bpm & -\phi_+^{-1} \cr \phi_- \epm$\,.
Note that $A$ is a complex structure on $F\oplus F$, that is $A^2 = -\id$. This can be deformed
into $\id$ using the homotopy $A_t = (\cos t)\,\id + (\sin t)A$ with $A_0 = \id$ and $A_{\pi/2} = A$.
By Remark \ref{homopm}, $\hat\phi'$ can be deformed onto an isomorphism $\tilde\phi' : \mathscr{E} \oplus\mathscr{F} \to \mathscr{F}\oplus\tilde{\mathscr{E}}$ with $\tilde\phi_+' = \id$.

However, we have yet to interchange the two summands of $\mathscr F \oplus \tilde{\mathscr E}$. 
Hence we apply the same construction once more, replacing $\phi : \mathscr E \to \tilde{\mathscr E}$ by $\phi''=\id : \tilde{\mathscr E} \to \tilde{\mathscr E}$.
This gives us a bundle isomorphism 
$\hat\phi'' = (A'',B'') : \tilde{\mathscr E}\oplus\mathscr F \to \mathscr F\oplus\tilde{\mathscr E}$
such that $A''$ can be deformed into $\id$, 
and by Remark \ref{homopm} again $\hat\phi''$ can be deformed into an isomorphism $\tilde\phi''$ of these bundles with $(\tilde\phi'')_+ = \id$. Then the composition 
$\tilde\phi:=(\tilde\phi'')^{-1}\tilde\phi' :  \mathscr{E} \oplus \mathscr F \to   \tilde{\mathscr{E}}\oplus \mathscr F$ satisfies $\tilde\phi_+ = \id$.
\end{proof}

Now, for the uniqueness part of Theorem \ref{main},  assume that ${E} , {\tilde E} \to X$ are Euclidean  vector bundles and ${\mathscr{L}}, {\tilde{\mathscr{L}}} \to \hat V$ the Hopf bundles corresponding to $\Cl(V)$-module bundles ${\Lambda}, {\tilde\Lambda}\to X$. 
Furthermore, for $ \mathscr{E} := {E} \oplus {\mathscr{ L}}$ 
and $ \tilde{\mathscr{E}}:={\tilde E} \oplus {\tilde{\mathscr{L}}}$ 
(after pulling back ${E}$ and ${\tilde E}$ to $\hat V$) we assume  
\begin{equation} \label{isomorphic}
    \mathscr{E}  \cong \tilde{\mathscr{E}} \, . 
\end{equation} 

Put $E_\pm = {E} \oplus {\Lambda}^\pm$ and 
$\tilde E_\pm = {\tilde E} \oplus {\tilde \Lambda}^\pm$.
We claim that after adding certain bundles $F,\tilde F\to X$ to $E_+,\tilde E_+$ with $F\cong_s\tilde F$  we may in addition assume that 
\beq \label{E+}
	E_+ = {\hat\Lambda}^+ = \tilde E_+
\eeq
for some $\Cl(V)$-module bundle ${\hat\Lambda} \to X$.
For proving this claim choose a $\Cl(V)$-module bundle $\Lambda_0 \to X$ such that ${E},{\tilde E}$ embed into  $\Lambda_0^+$ and put
${\Lambda = \Lambda\oplus \tilde \Lambda} \oplus \Lambda_0$. 
Put $F = {\tilde\Lambda^+\oplus(\Lambda_0^+\ominus E})$ and $\tilde F = {\Lambda^+ \oplus (\Lambda_0^+\ominus \tilde E})$. 
Then
\beq \label{FtildeF}
	F \oplus {\Lambda}^+ \oplus {E} = {\hat\Lambda}^+ = 
	\tilde F \oplus {\tilde\Lambda}^+ \oplus {\tilde E},
\eeq
but ${\Lambda^+}\oplus {E} = E_+$ and ${\tilde\Lambda}^+ \oplus {\tilde E} = \tilde E_+$ are isomorphic by \eqref{isomorphic} 
and thus $F\cong_s \tilde F$ by \eqref{FtildeF}, such that \eqref{E+} can be assumed.

\medskip
After this preparation we may hence assume that 
\bea
\mathscr{E}  &=& ({\hat\Lambda}^+ = {E\oplus \Lambda^+},\omega, {E\oplus \Lambda}^-)\,, \cr
\tilde{\mathscr{E}} &=& ({\hat\Lambda^+=\tilde E\oplus\tilde\Lambda}^+, \tilde \omega ,{\tilde E\oplus\tilde\Lambda}^-) \,,\nonumber
\eea
and 
\[
   \phi = (\phi_+, \phi_-) : \mathscr{E} \buildrel \cong\over\to \tilde{\mathscr{E}} 
\]
is a vector bundle isomorphism, which we may assume to be orthogonal.

Due to \eqref{E+}  Lemma \ref{isomorphism} implies that we can furthermore assume $\phi_+ = \id$, at least after adding some trivial bundle $\underline{\R}^q$ to $\mathscr{E}$ and $\tilde{\mathscr{E}}$ (chosen such that
the bundle $F$ in Lemma \ref{isomorphism} embeds as a subbundle of $\underline{\R}^q$). 
In particular this implies 
\[
   \tilde \omega = \phi_- \circ \omega  : \SS V \to \SO(\Lambda^+ , {\tilde E\oplus\tilde\Lambda}^- ) \, . 
\]
We have to show that 
\[
      {E} \cong_s {\tilde E} \, , \quad  {\Lambda} \cong_s {\tilde \Lambda} \, . 
 \]
 But the given data induce a triple $\hat{\mathscr{E}} =  ( {\Lambda^+}, \tau, \hat E_-)$ over  $\hat V \times [0,1]$ (the total space of $\hat V \times [0,1] \to X \times [0,1]$) by pulling back the triples $\mathscr{E}$ and $\tilde{\mathscr{E}}$ over $\hat V \times [0,1/2]$, respectively $\hat V \times [1/2, 1]$ and gluing them along $\hat V \times \{1/2\}$ by means of the isomorphism $(\id,\phi)$. 
More precisely, this is done as follows. {Recall}
\[
      E_- = {E} \oplus \Lambda^- \, , \quad \tilde E_- = {\tilde E} \oplus \tilde\Lambda^- \, . 
 \]
We have the isomorphism $\phi_- : E_- \to \tilde E_-$ over $\D_-V$ with $\phi\o\omega = \tilde\omega$ over $\SS V$. 
Let $s_-$ be the 0-section of $\D_-V$ and $\phi_o = \phi_{-}|_{s_-}$, which is an isomorphism between $E_-$ and $\tilde E_-$ over $X$. 
Note that $\phi_o$ and  
$\phi_-$ are homotopic. 

We define a bundle $\hat E_- \to X \x [0,1]$ as follows. 
We put
$$
	\hat E_- = \big(E_-\x [0,1/2]\big)\cup_{\phi_o}\big(\tilde E_-\x [1/2,1]\big)
$$
using the isomorphism $\phi_o$ to identify $E_-\x\{\2\}$ with $\tilde E_-\x\{\2\}$.
This defines a vector bundle $\hat E_-\to X\x[0,1]$.
Further, we define a clutching map 
\[
     \hat\sigma : {\hat\Lambda}^+\x[0,1] \to \hat E_-
\]
over $\SS V \x [0,1]$ as follows.
$$
\hat\sigma(v,t) = \left\{\begin{matrix} \mu(v) & \textrm{for} & 0\leq t\leq 1/2  , \cr
			\phi_t\o\mu(v) &\textrm{for} & 1/2\leq t\leq 3/4 ,  \cr
				\tilde\mu(v) &\textrm{for} & 3/4\leq t\leq 1,  \cr
\end{matrix}\right.
$$
where $\phi_t: E_- \to \tilde E_-$ is a homotopy of bundle isomorphisms over $\D_-V$ for $1/2\leq t\leq 3/4$ with $\phi_{1/2} = \phi_o$ and $\phi_{3/4} = {\phi_-}$.

Then $\hat\sigma$ is well defined at $t=1/2$ since for any $\xi \in {\hat\Lambda}_x$ and $v\in \SS V_x$, the element $\mu(v)\xi\in (E_-)_x$ is identified with $\phi_o\mu(v)\xi \in (\tilde E_-)_x$.
It is also well defined at $t=3/4$ because $\phi\o\mu = \tilde\mu$, see \eqref{isomorph} with $\phi_+ = \id$. 

Thus we obtain a triple  
\[
    \hat{\mathscr{E}} = (\Lambda^+\x[0,1],\hat\sigma,\hat E_-)  
\]
over $\hat V\x[0,1]$ which restricts to $\mathscr{E}$ on $\hat V \times \{0\}$ and to $\tilde{\mathscr{E}}$  on $\hat V \times \{1\}$. 

We now apply the previous deformation process in the proof of Equation \eqref{hopftriple} in Proposition \ref{important} on $X\x[0,1]$, 
but relative to $X \times \{0,1\}$, i.e.\ we start the induction at $X' = X \times \{0,1\}$. 
Now we have to assume $d = \dim X + 1$ in {Equation} \eqref{d,p}. 
This implies that (stably) the vector bundle ${E}  \sqcup {\tilde E} \to X \times \{0,1\}$ {(called $\Lambda_0^+$ in \ref{important})} extends over $X \times [0,1]$,  and the analogous holds for the $\Cl(V)$-module bundle  ${\Lambda} \sqcup {\tilde \Lambda} \to X \times \{0,1\}$ {(called $\Lambda_1$ in \ref{important})}. 

Hence we obtain stable isomorphisms ${E} \cong_s {\tilde E}$ and ${\Lambda} \cong_s {\tilde \Lambda}$, finishing  the proof of the uniqueness statement in Theorem \ref{main}.

\section{Thom isomorphism theorems} \label{CliffordThom} 

Theorem \ref{main} can be reformulated concisely in the language of topological K-theory. 
Let $X$ be a finite CW-complex and $V \to X$ an oriented Euclidean vector bundle of rank $n = 4m \geq 0$ with associated Clifford algebra bundle $\Cl(V) \to X$. 

\begin{defn}  \label{KCL} 
We denote by  $\KK^{\Cl(V)}(X)$ the topological $\Cl(V)$-linear K-theory of $X$.
More precisely, elements in $\KK^{\Cl(V)}(X)$ are represented by formal differences of isomorphism classes  of $\Cl(V)$-module bundles  (cf. Example \ref{Lambda}), with  addition induced by the direct sum  and neutral element the trivial $\Cl(V)$-module bundle with fibre $0$. 
\end{defn} 

Each $\Cl(V)$-module bundle $\Lambda \to X$  is isomorphic to a $\Cl(V)$-submodule bundle of $\Cl(V)^q \to X$ for some $q$.
This holds in the special case $X = \{{\rm point}\}$, since each $\Cl_n$-module  is a direct sum of irreducible  $\Cl_n$-modules, each of which also occurs as a summand in  the  $\Cl_n$-module $\Cl_n$. 
\footnote{Recall that by Theorem \ref{Mk} there is, up to isomorphism, just one irreducible $\Cl_n$-module for $n = 4m$.} 

For more general $X$ we use a partition of unity subordinate to a finite cover of $X$ by open trivializing subsets for the given $\Cl(V)$-module bundle, similar as for ordinary vector bundles. 

In particular,  two $\Cl(V)$-module bundles $\Lambda, \tilde \Lambda$ represent the same element in $\KK^{Cl(V)}(X)$, if and only if  they are stably isomorphic as $\Cl(V)$-module bundles in the sense of Definition \ref{stabiso}(2).

Let $\KK^{\Or}(X)$ denote the orthogonal  topological K-theory of $X$. 
If $X$ is equipped with a base point $x_0$ recall the definition of the reduced real K-theory
\[
    \tilde \KK^{\Or}(X) = \ker ( \KK^{\Or}(X) \stackrel{restr.}{\longrightarrow} \KK^{\Or} (x_0)) \subset   \KK^{\Or}(X) \, . 
\]
 Let  $X^V = \hat V /  s_+ $ be the Thom space associated to $V\to X$ with base point $[s_+]$ (recall that  $s_+ \subset \hat V$ denotes the zero section of $\D_+ V$). 
The fibrewise projection $\pi : \hat V \to s_+ $ induces a group homomorphism 
\begin{equation} \label{gotoreduced} 
   \chi : \KK^{\Or}(\hat V) \to \tilde \KK^{\Or}(X^V) \, , \quad [E] \mapsto [E] - \pi^*[E|_{s_+}] \, . 
\end{equation} 

\begin{defn}  \label{def:CliffThom} 
The group homomorphism  
\[
    \Psi_V :    \KK^{\Cl(V)} (X)  \to  \tilde \KK^{\Or}(X^V) \, , \quad \lbrack \Lambda\rbrack  \mapsto  \chi  ([\mathscr{L}]) \, , 
\]
where  $\mathscr{L} =(\Lambda^+, \mu , \Lambda^-) \to \hat V$ is the Hopf bundle associated to $\Lambda$ (see Example \ref{Lambda}),  is called  the {\em Clifford-Thom homomorphism}. 
\end{defn} 

In this language Theorem  \ref{main} translates to the following  result.

\begin{thm} \label{clthom} 
The Clifford-Thom homomorphism $\Psi_V$  is an isomorphism. 
\end{thm} 

We will now point out that Theorem \ref{clthom} implies classical Thom  isomorphism theorems for orthogonal, unitary and symplectic  K-theory (the orthogonal and unitary cases are also treated in  \cite[Theorem IV.5.14]{KaroubiBook}). 

Assume that $V \to X$ is equipped with a spin structure and let $\Sigma \to X$ be the spinor bundle associated to $V \to X$, compare Example \ref{spinor}. 
In this case $\Cl(V)$-module bundles are of a particular form.

Recall the notion of the {\it tensor product over $\H$} (e.g.\ see \cite[p.\,18]{D}):
Let $P$ a $\H$-right vector space and $Q$ a $\H$-left vector space. 
Then
\bea
	P\otimes_{\H} Q &:=& (P\otimes Q)/R\,, \textrm{ where }\cr
	 R &:=& {\rm Span}\,\{p\lambda\otimes q-p\otimes\lambda q: p\in P, q\in Q, \lambda\in\H\}.
\nonumber
\eea
This is a vector space over $\R$, not over $\H$.

If $m$ is odd then  the right $\H$-multiplication on $S$ commutes with the left $\Cl_{n}$-action (see Table \eqref{Mk}), and hence $\Sigma$ is a bundle of right $\H$-modules in a canonical way. 
We equivalently regard $S$ and $\Sigma$ as left $\H$-modules by setting $\lambda \cdot s := s \lambda^{-1}$.

\begin{prop} \label{meven}
Let $\Lambda$ be a $\Cl(V)$-module bundle. 
Then $\Lambda$ is a twisted spinor bundle, that is there exists a Euclidean vector bundle $E\to X$ over $\R$ or $\H$ when $m$ is even or odd,
respectively, such that
\beq \label{EoxSigma}
	\Lambda \cong \left\{\begin{matrix}
	E \otimes \Sigma     & \textrm{when $m$ is even } \cr
	E \otimes_{\H} \Sigma & \textrm{when $m$ is odd }
\end{matrix}\right.
\eeq
If two such bundles $\Lambda = E\otimes_{(\H)} \Sigma$ and  $\tilde\Lambda = \tilde E\otimes_{(\H)}\Sigma$ are isomorphic as $\Cl(V)$-module bundles, then $E\cong\tilde E$.
\end{prop}

\proof
Following \cite[p.\,115]{BGV}, we put
$$
	E = \Hom_{\Cl(V)}(\Sigma,\Lambda).
$$
The isomorphism in \eqref{EoxSigma} is given by
\beq \label{jdef}
	j : E\otimes\Sigma\to\Lambda, \phi\otimes\xi \mapsto \phi(\xi).
\eeq
This is a $\Cl(V)$-linear bundle homomorphism: For all $\alpha\in \Cl(V)$,
$$
j(\alpha\.( \phi\otimes\xi) ) = j(\phi\otimes(\alpha\.\xi))  = \phi(\alpha\.\xi) = \alpha\.\phi(\xi) = \alpha\.j(\phi\otimes\xi).
$$
To check bijectivity we look at the fibers $L$ of $\Lambda$ and $S$ of $\Sigma$.
 Set $C := \Cl_n$. 
 The fibre of  $E = \Hom_{\Cl(V)}(\Sigma,\Lambda)$ is $\Hom_{C} (S,L)$ and the homomorphism $j$ is fiberwise the linear map 
$$
	j_o : \Hom_C(S,L)\otimes S \to L, \ (\phi,s)\mapsto \phi(s).
$$
Since $S$ is the unique irreducible $C$-representation we have $L \cong S^p = \R^p\otimes S$ for some $p$, and $\Hom_C(S,L) = \Hom_C(S,S)^p = \End_C(S)^p$.
Applying $j_o$ to $\phi = (0, \ldots, \id_S, \ldots, 0)$ with the identity at the $k$-th slot for $k=1,\dots,p$, we see that this map is onto. 
When $m$ is even we have $C = \End_{\R}(S)$, and when $m$ is odd, $C = \End_{\H}(S)$ (cf.\ Table \eqref{Mk}). 
Thus $\End_C(S) = \R\.\id_S$ when $m$ is even and $\End_C(S) = \H\.\id_S$  when $m$ is odd.
Note that in the second case an element $\lambda \in \H$ corresponds to  $\hat\lambda := R_{\lambda^{-1}} \in  \End_C(S)$ (right multiplication with $\lambda^{-1}$ on $S$, that is left multiplication with $\lambda$ with respect to the left $\H$-module structure defined before) in order to make the identification a ring map.\footnote 
	{$\End_C(S)$ is an (associative) division algebra, 
	thus isomorphic to $\H$ or $\C$ or $\R$. In fact, 
	kernel and image of any $A\in\End_C(S)$
	are $C$-invariant subspaces of the irreducible $C$-module $S$, hence $A$ is
	invertible or $0$. When $m$ is odd, $C = \End_{\H}(S)$ and therefore
	$\End_C(S)\supset\H$, hence $\End_C(S) = \H$. 
	When $m$ is even, $C = \End_{\R}(S)$.  
	Any $A\in\End_C(S)$ has a real or complex eigenvalue 
	and hence an invariant line or plane in $S$. 
	Since it commutes with all endomorphisms on $S$, {\em every} line or plane is
	$A$-invariant since $GL(S)\subset\End(S)$ acts transitively on the 
	Grassmannians. Thus $A$ is a real multiple of the identity. 
	(See also Wedderburn's theorem.)}

In the first case, $\Hom_C(S,L) = \End_C(S)^p\otimes S = \R^p\otimes S = S^p$. 
Thus $j_o$ is an isomorphism whence $j$ is an isomorphism.

In the second case we consider $\Hom_C(S,L)$ as a $\H$-right vector space using pre-composition $\phi\mapsto \phi\o\hat\lambda$, where $\hat\lambda = R_{\lambda^{-1}} \in \End_C(S) = \H$.
We hence compute 
\[ 
      j_o(\phi\lambda,\xi) = ( \phi \circ \hat\lambda ) ( \xi ) =  \phi ( \xi \lambda^{-1}) = \phi ( \lambda \cdot \xi) = j_o(\phi,\lambda  \cdot \xi) \, . 
\]
Thus $j_o$ descends to $\bar j_o : \Hom_C(L,S)\otimes_{\H}S$. 
Since $\Hom_C(L,S) \cong \Hom_C(S^p,S) = \End_C(S)^p = \H^p$ and $\H\otimes_{\H}S \cong S$ via the map  $\lambda\otimes_{\H}s = 1\otimes_{\H}\lambda s\mapsto \lambda s$, we obtain that $\bar j_o$ maps $\Hom_C(S,L)\otimes_{\H} S \cong \H^p\otimes_{\H}S = (\H\otimes_{\H}S)^p \cong S^p$ isomorphically onto $L\cong S^p$.

A $\Cl(V)$-linear isomorphism $\phi$ between two such bundles $E\otimes_{(\H)}\Sigma$ and  $\tilde E\otimes_{(\H)} \Sigma$ is fiberwise a $C$-linear isomorphism $\phi_o : S^p\to S^p$. 
This is an invertible $(p\x p)$-matrix $A = (a_{ij})$ with coefficients in $\End_C(S) = \K$ where $\K = \R$ when $m$ is even and $\K = \H$ when $m$ is odd. 
In the first case we have $S^p = \R^p\otimes S$ and $j_o = A\otimes\id_S$. 
In the second case we let $S^p = \H^p\otimes_{\H}S$ using the isomorphism $se_i \mapsto e_i\otimes_{\H}s$ for all $s\in S$ and $i=1,\dots,p$  where $e_1 = (1,0,\dots,0)$, $\dots$, $e_p = (0,\dots,0,1)$.
Then 
\[
   \phi_o(se_i) = \sum_j a_{ij}se_j\mapsto \sum_j e_j\otimes_{\H}a_{ij}s = \sum_j e_ja_{ij}\otimes_{\H} s = Ae_i\otimes_{\H}s
\]
and hence $\phi_o = A\otimes_{\H}\id_S$.

Therefore $\phi = f\otimes \id_{\Sigma}$ for some $\R$-linear isomorphism
$f : E\to\tilde E$ in the first case, and $\phi = f\otimes_{\H}\id_\Sigma$
for some $\H$-linear isomorphism $f : E\to\tilde E$ in the second case.
\endproof

\begin{lem} \label{95}
Let $\Lambda , \tilde \Lambda \to X$ be $\Cl(V)$-module bundles. 
Then the following assertions are equivalent. 
\begin{enumerate}[(i)]  
   \item  $\Lambda$ and $\tilde \Lambda$ are stably isomorphic as $\Cl(V)$-module bundles. 
   \item There is some $q \geq 0$, such that $\Lambda \oplus \Sigma^q \cong \tilde \Lambda \oplus \Sigma^q$ are isomorphic $\Cl(V)$-module bundles 
   \end{enumerate} 
   \end{lem} 
   
\begin{proof} This follows since $\Sigma \to X$ is a $\Cl(V)$-submodule bundle of $\Cl(V)^q \to X$ for some $q \geq 0$, and vice versa $\Cl(V)\to X$ is a $\Cl(V)$-submodule bundle of $\Sigma^q \to X $ for some $q$. 
Again this is obvious if $X$ is equal to a point and follows for general $X$ by a partition of unity argument. 
\end{proof} 

Together with  Proposition \ref{meven} this implies 

\begin{prop}  \label{96}
Let $V \to X$ be of rank $4m$ 
and equipped with a spin structure, and let $\Sigma \to X$ denote the associated spinor bundle. 
\begin{enumerate}[(a)] 
\item For even $m$ the map 
\[
     \KK^{\Or} (X) \to \KK^{\Cl(V)}(X) \, , \quad \lbrack E \rbrack \mapsto \lbrack E \otimes \Sigma \rbrack \, , 
\]
is an isomorphism. 
\item For odd $m$ the map 
\[
     \KK^{{\rm Sp}} (X) \to \KK^{\Cl(V)}(X) \, , \quad \lbrack E \rbrack \mapsto \lbrack E \otimes_{\H}  \Sigma \rbrack \, , 
\]
is an isomorphism, where $\KK^{\rm Sp}$ denotes symplectic K-theory based on $\H$-right vector bundles.
\end{enumerate} 
\end{prop} 

\proof
The maps in (a) and (b) are well defined and one-to-one: Let $E,\tilde E \to X$ be vector bundles.
Then $\Lambda = E\otimes\Sigma$
is stably isomorphic to $\tilde\Lambda = \tilde E\otimes\Sigma$ $\buildrel \ref{95}\over\iff$
$\Lambda \oplus \Sigma^q=(E\oplus\underline{\R}^q)\otimes\Sigma$ (for some $q\in\N$ given by \ref{95}) is isomorphic to
$\tilde\Lambda\oplus\Sigma^q=(\tilde E\oplus\underline{\R}^q)\otimes\Sigma$
$\buildrel \ref{meven}\over\iff$ $E\oplus\underline{\R}^q\cong\tilde E\oplus \underline{\R}^q$ $\iff$ 
$E\cong_s\tilde E$. Further, these maps are
onto by Prop.\ \ref{meven}.
\endproof

Hence we obtain the classical Thom isomorphism  by composing the isomorphisms of \ref{96} and \ref{clthom}, using that the spinor Hopf bundle $\mathscr S$ (cf.\ \ref{spinorhopf}) is the Hopf bundle (cf.\ \ref{clrepsphere}) associated to the spinor bundle $\Sigma$:

\begin{thm} \label{ThomClassic} Let $V \to X$ be equipped with a spin structure and let $\mathscr{S}\to\hat V$ denote the spinor Hopf bundle associated to $V \to X$ (see Definition \ref{spinorhopf}). 
\begin{enumerate}[(a)] 
\item For even $m$  the map 
\[
     \KK^{\Or} (X) \to \tilde  \KK^{\Or}(X^V) \, , \quad \lbrack E \rbrack \mapsto \chi (  [ E \otimes \mathscr{S} ] ) , 
\]
is an isomorphism. 
\item For odd $m$ the map 
\[
     \KK^{{\rm Sp}} (X) \to \tilde  \KK^{\Or}(X^V) \, , \quad \lbrack E \rbrack \mapsto \chi (  [ E \otimes_{\H} \mathscr{S}] ) , 
\]
 is an isomorphism. 
\end{enumerate} 
Hence $\chi ( [ \mathscr{S}])  \in  \tilde  \KK^{\Or}(X^V)$ serves as the ``$\KK^{\Or}$-theoretic Thom class'' of $V$. 
\end{thm}

There are analogues of the theorems \ref{main} and \ref{clthom} for complex and quaternionic vector bundles over $\hat V$. Detailed proofs of the following statements are left to the reader.

\begin{thm} \label{allg}  Let $X$ be a finite CW-complex and $V \to X$ an oriented Euclidean vector bundle with associated Clifford algebra bundle $\Cl(V) \to X$. 
  \begin{enumerate}[(a)] 
     \item Let $\rk V = 4m$.
             Then there is a Clifford-Thom isomorphism 
      \[
         \KK^{\Cl(V) \otimes \H} (X)  \cong  \tilde \KK^{\Sp}(X^V).
    \]
    \item Let $\rk V = 2m$. 
            Then there is  a Clifford-Thom isomorphism 
    \[
         \KK^{\Cl(V) \otimes \C} (X) \cong \tilde \KK^{\U} (X^V).
    \]
 \end{enumerate} 
 \end{thm}

The analogues of Proposition \ref{meven} for complex and quaternionic $\Cl(V)$-module bundles are as follows. 

Assume that $V \to X$ is equipped with a spin structure. 
Recall from \eqref{Mk} that the spinor bundle $\Sigma$ associated to $P_{\Spin}(V)$ is quaternionic when $\rk V = n = 4m$ with $m$ odd and it is real when $n = 4m$ with $m$ even. 
If $\Lambda$ is a $(\Cl(V)\otimes \H)$-module bundle (with left $\Cl(V)$-multiplication and right $\H$-multiplication), we put $E = \Hom_{\Cl(V) \otimes \H}(\Sigma,\Lambda)$ for $m$ odd and $E = \Hom_{\Cl(V)}(\Sigma,\Lambda)$ if $m$ is even; in both cases precisely one of the bundles $E$ and $\Sigma$ is (right) quaternionic, and we obtain an isomorphism of $(\Cl(V) \otimes \H)$-module bundles
analogue to \eqref{jdef}, 
$$
	j : E\otimes_{\R}\Sigma \to \Lambda,\ \ \phi\otimes\xi\mapsto\phi(\xi).
$$

For $(\Cl(V)\otimes\C)$-module bundles $\Lambda$ with $\dim V = n = 2m$  we can apply \cite[Prop.\,3.34]{BGV}:
Let  $\Sigma^c$ be the complex spinor bundle associated to $P_{\Spin^c}(V)$ by means of a Spin$^c$-structure on $V$ (cf. \cite[Appendix D]{LM}).
Putting $E = \Hom_{\Cl(V) \otimes \C}(\Sigma^c,\Lambda)$ we obtain that
$$
	j : E\otimes_{\C}\Sigma^c\to\Lambda,\ \ \phi\otimes_{\C}\xi\mapsto\phi(\xi)
$$
is an isomorphism of $(\Cl(V)\otimes\C)$-module bundles.

Together with Theorem \ref{allg} we hence arrive at the following classical Thom isomorphism theorems, which complement Theorem \ref{ThomClassic}.

\begin{thm} \label{final} 
\begin{enumerate}[(a)]  
\item Let $V \to X$ be a spin bundle of rank $n = 4m$. 
Then multiplication with the spinor Hopf bundle $\mathscr{S} \to \hat V$ induces  Thom isomorphisms 
\begin{align*} 
     \KK^{\Sp} (X) & \to \tilde  \KK^{\Sp}(X^V)  \text{ for  even } m\, , \\ 
     \KK^{\Or} (X) &  \to \tilde  \KK^{\Sp}(X^V) \text{ for odd } m \, . 
\end{align*} 
 \item Let $V \to X$ be a $\Spin^c$-bundle of rank $n = 2m$. 
    Then multiplication with the complex spinor Hopf bundle $\mathscr{S}^c \to \hat V$ induces  a Thom isomorphism  
\[
     \KK^{\U} (X) \to \tilde  \KK^{\U}(X^V) \, . 
\]
\end{enumerate} 
\end{thm}

\begin{disc} \label{Discussion} 
{We will point out some connections of the argument at the beginning of this section to the classical monograph \cite{KaroubiBook}.
Let $Q$ denote the given (positive definite)  quadratic form on the Euclidean bundle $V \to X$. 
Then the  Clifford algebra bundle $\Cl(V) \to X$ in the sense of Definition \ref{Lambda} is equal to the  Clifford algebra bundle $C(V, - Q) \to X$ in the sense of  \cite[IV.4.11]{KaroubiBook}  \footnote{Notice the sign convention for the construction of the Clifford algebras  in \cite[III.3.1]{KaroubiBook}, which is different from ours.} and $\KK^{\Cl(V)}(X)$ from  Definition \ref{KCL}   is equal to $\KK( \mathscr{E}^{(V,-Q)} (X))$, the Grothendieck group of the Banach category $\mathscr{E}^{(V,-Q)} (X)$  of $C(V,-Q)$-module bundles over $X$, compare \cite[II.1.7]{KaroubiBook}. 

Since $V$ is oriented and of rank divisible by four we have $C(V,Q) \cong C(V,-Q)$ as algebra bundles, induced by the map $V \to C(V,Q)$, $v \mapsto \omega \cdot v$, where $\omega \in \Cl(V)$ denotes the volume section of $C(V,-Q)$ (note that $v^2 = - \|v\|^2$ and $4 \mid \rk  V$ imply $(\omega \cdot v)^2 = +\|v\|^2$). 
Hence 
\begin{equation} \label{diffid} 
    \KK( \mathscr{E}^{(V,Q)}(X)) \cong   \KK( \mathscr{E}^{(V,-Q)} (X)) =  \KK^{\Cl(V)}(X)  
\end{equation} 
 for the given $V \to X$.

In the following we denote by   $C(V) \to X$ the algebra bundle $C(V,Q) \to X$ and $\mathscr{C} = \mathscr{E}^V(X)$ the Banach category of module bundles over $C(V) \to X$. 

As in \cite[Section IV.5.1]{KaroubiBook} let $\KK^V(X) = \KK(\phi^V) $ be  the Grothendieck group of the forgetful functor $\phi^V : \mathscr{E}^{V \oplus \R}(X) \to \mathscr{E}^V(X)$. 
Using $C(\R, Q) = \R \oplus \R$ for the standard quadratic form $Q$ on $\R$, compare \cite[III.3.4]{KaroubiBook}, we have $C(V \oplus \R) \cong C(V) \otimes ( \R \oplus \R) = C(V) \oplus C(V)$, see \cite[Prop. III.3.16.(i)]{KaroubiBook} (note that $C(V,Q) > 0$, which means $\omega^2 = + 1$, holds since $4 \mid \rk V$).  

Hence $\mathscr{E}^{V \oplus \R}(X) \cong \mathscr{C} \times \mathscr{C}$ and  the forgetful functor $\phi^V$ can be identified with the functor $\psi: \mathscr{C} \times \mathscr{C}  \to \mathscr{C}$, sending a pair of $C(V)$-module bundles $(E, F)$ to their direct sum $E \oplus F$, compare \cite[III.4.9]{KaroubiBook} for the cases $q = 0$ and $q =  4$. 
Since this functor has a right inverse $E \mapsto (E,0)$ the first and last maps in the exact sequence 
\[
     \KK^{-1}( \mathscr{C} \times \mathscr{C} )  \stackrel{\psi_*}{\longrightarrow}  \KK^{-1}( \mathscr{C}) \to \KK(\phi^V) \to \KK( \mathscr{C} \times \mathscr{C}) \stackrel{\psi_*}{\longrightarrow}  \KK( \mathscr{C} ) 
\]
from \cite[II.3.22]{KaroubiBook} are surjective. 
We hence obtain isomorphisms 
\[
    \KK(\phi^V)  \cong \ker \left(\psi_* :  \KK( \mathscr{C}) \oplus \KK(\mathscr{C}) \stackrel{(a,b) \mapsto a+ b}{\longrightarrow} \KK( \mathscr{C} ) \right) \cong \{ (a,-a) \mid a \in \KK( \mathscr{C} ) \} \cong \KK(\mathscr{C}) 
\]
showing that 
\[
    \KK^V(X) \cong  \KK(\phi^V) \cong \KK(\mathscr{C}) \stackrel{\eqref{diffid}}{\cong} \KK^{\Cl(V)} (X) \, . 
\]
Put differently: For the given $V \to X$ Karoubi's theory $\KK^V(X)$  is the topological K-theory based on $\Cl(V)$-module bundles over $X$. 
We remark that this is not true in general for  bundles $V \to X$ of rank not divisible by four.

One can show  that with respect to this identification the Clifford-Thom homomorphism in Definition \ref{def:CliffThom} is identified with the map $t : \KK^V(X) \to \KK(B(V), S(V))$ from \cite[IV.5.10.]{KaroubiBook}. 
Hence Theorem \ref{clthom} recovers \cite[Theorem IV.5.11]{KaroubiBook} (whose full proof is given in \cite[IV.6.21]{KaroubiBook}), for oriented $V \to X$ of rank divisible by four.}
\end{disc}

\end{document}